\documentclass[10pt,reqno]{amsart}

\usepackage[utf8]{inputenc}
\usepackage[english]{babel}
\usepackage{amsmath,amsfonts,amssymb,amsthm,shuffle}
\usepackage[T1]{fontenc}
\usepackage{lmodern}
\usepackage{mathtools}
\usepackage[titletoc]{appendix}
\usepackage[foot]{amsaddr}

\usepackage[top=3.5cm,bottom=3.5cm,left=3.6cm,right=3.6cm]{geometry}

\usepackage[dvipsnames]{xcolor}
\usepackage[hyperindex=true,frenchlinks=true,colorlinks=true,
citecolor=Mahogany,linkcolor=DarkOrchid,urlcolor=Tan,linktocpage,
pagebackref=true]{hyperref}

\usepackage{tikz}
\usetikzlibrary{shapes}
\usetikzlibrary{fit}
\usetikzlibrary{decorations.pathmorphing}

\usepackage{dsfont}
\usepackage{wasysym}
\usepackage{stmaryrd}
\usepackage{cite}
\usepackage{subfigure}
\usepackage{multirow}
\usepackage{enumitem}
\usepackage{multicol}
\usepackage{bm}

\title[]{Maximal function associated to the bounded 
law of the iterated logarithms via orthomartingale 
approximation}
\date{\today}
\author[Davide Giraudo]{Davide 
Giraudo \\
davide.giraudo@rub.de    }
\numberwithin{equation}{subsection}
\renewcommand{\leq}{\leqslant}
\renewcommand{\geq}{\geqslant}

\newtheorem{Theorem}{Theorem}[section]
\newtheorem{Th\'eor\`eme}{Th\'eor\`eme}[section]
\newtheorem{Proposition}[Theorem]{Proposition}
\newtheorem{Lemma}[Theorem]{Lemma}

\newtheorem{Definition}[Theorem]{Definition}
\newtheorem{D\'efinition}[Th\'eor\`eme]{D\'efinition}
\newtheorem{Corollary}[Theorem]{Corollary}

\theoremstyle{remark}
\newtheorem{Remark}[Theorem]{Remark}

\tikzstyle{Vertex}=[circle,draw=LimeGreen!80,fill=LimeGreen!8,
inner sep=1pt,minimum size=2mm,line width=1pt,font=\scriptsize]
\tikzstyle{Node}=[Vertex,draw=RoyalBlue!80,fill=RoyalBlue!8,inner sep=1.5pt]
\tikzstyle{Leaf}=[rectangle,draw=Black!70,fill=Black!16,
inner sep=0pt,minimum size=1mm,line width=1.25pt]
\tikzstyle{Edge}=[Maroon!80,cap=round,line width=1pt]
\tikzstyle{Mark1}=[draw=BrickRed!80,fill=BrickRed!8]
\tikzstyle{Mark2}=[draw=BurntOrange!80,fill=BurntOrange!8]
\tikzstyle{EdgeRew}=[->,RedOrange!80,cap=round,thick]

\newcommand{\Aca}{\mathcal{A}}

\newcommand{\Fca}{\mathcal{F}}
\newcommand{\Gca}{\mathcal{G}}

\newcommand \ens[1]{\left\{ #1\right\}}
\newcommand \R{\mathbb R}

\newcommand \N{\mathbb N}
\newcommand \PP{\mathbb P}
\newcommand{\el}{\mathbb L}
\newcommand{\E}[1]{\mathbb E\left[#1\right]}
\newcommand{\proj}[2]{\mathbb{E}_{#2}\left[#1\right]}

\newcommand \Z{\mathbb Z}

\newcommand \abs[1]{\left|#1\right|}

\newcommand{\f}{\mathcal F}

\newcommand{\pr}[1]{\left(#1\right)}
\newcommand{\norm}[1]{\left\lVert #1 \right\rVert}
\newcommand{\gr}[1]{\bm{#1}}
\newcommand{\gri}{\bm{i}}
\newcommand{\grj}{\bm{j}}
\newcommand{\grn}{\bm{n}}
\newcommand{\grm}{\bm{m}}

\newcommand{\grk}{\bm{k}}
\newcommand{\imd}{\preccurlyeq}
\newcommand{\smd}{\succcurlyeq }

\begin{document}
\begin{abstract}
We give sufficient conditions for the bounded law of the iterated logarithms for
strictly stationary random fields when the summation is done on rectangle. The study is
done by the control of an appropriated maximal function. The case of  orthomartingales
is treated. Then results on projective conditions are derived.
\end{abstract} 
\maketitle

\section{Introduction, goal of the paper}

 \subsection{Bounded law of the iterated logarithms for random fields} 
Before we present the scope of the paper, let us introduce the 
following notations.

\begin{enumerate}
 \item In all the paper, $d$ is an integer greater or equal to one. 
 \item For any integer $N$, we denote by $[N]$ the set 
 $\ens{k\in\Z, 1\leq k\leq N}$.
 \item The element of $\Z^d$ whose coordinates are 
 all $0$ (respectively 
$1$) is denoted by $\gr{0}$ (resp. $\gr{1}$).

\item We denote by $\imd$ the coordinatewise order on 
the elements of $\Z^d$, that is, we write for $\gri=\pr{i_q}_{q=1}^d$ 
and $\grj=\pr{j_q}_{q=1}^d$ that 
$\gri\imd\grj$ if $i_q\leq j_q$ for all $q\in [d]$. Similarly, 
we write $\gri\smd\grj$ if $i_q\geq j_q$ for all $q\in [d]$.

\item For a family of numbers $\pr{a_{\grn}}_{\grn\smd \gr{1}}$, we define 
$\limsup_{\grn\to +\infty}a_{\grn}:=\lim_{m\to +\infty}\sup_{\grn\smd m\gr{1}}a_{\grn}$.

\item Let $L\colon \pr{0,+\infty}\to \R$ be defined by 
$L\pr{x}=\max\ens{\ln x,1}$ and $LL\colon \pr{0,+\infty}\to \R$ 
by $LL\pr{x}=L\circ L\pr{x}$.
\end{enumerate}

Let $\pr{X_{\gri}}_{\gri\in\Z^d}$ be a strictly stationary random 
field and denote for $\grn\smd \gr{1}$ the partial sum 
\begin{equation}
 S_{\grn}:=\sum_{\gr{1}\imd\gri \imd\grn}X_{\gri}.
\end{equation}

We are interested in finding a family of positive numbers 
$\pr{a_{\grn}}_{\grn\smd\gr{1}}$ with the smallest possible
 growth as $\max \grn:=\max_{1\leq q\leq d}n_q\to\infty$ such that the quantitiy 
\begin{equation}\label{eq:dnf_supremum_somme_normalisee}
\norm{\sup_{\grn\smd\gr{1}}\frac 1{a_{\grn}}
\abs{S_{\grn}}}_p<+\infty, 1\leq p<2,
\end{equation}
is finite.
It has been shown in \cite{MR0394894}
that for an i.i.d. collection of centered random variables 
$\ens{X_{\gri},\gri\in\Z^d}$, (with $d>1$) satisfying 
$\E{X_{\gr{0}}^2\pr{L\pr{\abs{X_{\gr{0}}}     }     }^{d-1}/LL
\pr{\abs{X_{\gr{0}}}   }}<+\infty$, then 
\begin{equation}
\limsup_{\grn \to +\infty}\frac{1}{\sqrt{\abs{\grn}L L\pr{\abs{\grn}    }  }}
S_{\grn}=\norm{X_{\gr{0}}}_2\sqrt d=
-\liminf_{\grn \to +\infty}\frac{1}{\sqrt{\abs{\grn}L L\pr{\abs{\grn}    }  }}
S_{\grn}.
\end{equation}
In particular, the moment condition as well as the $\limsup$/$\liminf$ depend on the 
dimension $d$ and the choice of $a_{\grn}=\sqrt{\abs{\grn}  LL\pr{\abs{\grn}}  }$ 
is the best possible among those guaranting the finiteness of the random variable 
involved in \eqref{eq:dnf_supremum_somme_normalisee}. 

In this paper, we will be concentrated in the following questions. First, 
we would like to give bound on the quantity
involved in \eqref{eq:dnf_supremum_somme_normalisee}. Results 
in the one dimensional case are known in the i.i.d. setting \cite{MR0501237} and martingales \cite{MR3322323}, but to the best of
 our knowledge, it seems that no results are available in dimension 
 greater than one. Nevertheless, the question of giving the limiting points of $\pr{S_{\grn}/\sqrt{\abs{\grn}LL\pr{\abs{\grn}}}  }_{\grn \smd \gr{1}}$ 
 has been investigated in dimension 2 in \cite{MR1674964}.
  
A first objective is to deal with the case of orthomartingales. Approximations by the latter class of random fields 
lead to results for the central limit theorem and 
its functional version (see \cite{MR3504508,MR3869881, MR3798239,
 MR3769664}). Therefore, a reasonable objective is to establish similar 
 results in the context of the bounded law of the iterated logarithms. Therefore, the 
 second objective is to deal with projectives conditions in order to extend the results for orthomartingales 
 to larger classes of random fields.

\subsection{Stationary random fields}

\begin{Definition}
We say that the random field $\pr{X_{\gri}}_{\gri\in \Z^d}$ is strictly stationary if 
for all $\grj \in \Z^d$, all $N\geq 1$ and all $\gr{i_1},\dots, \gr{i_N}$, the vectors 
$\pr{X_{\gr{i_1}+\grj  },\dots,X_{\gr{i_N}+\grj }}$ and 
$\pr{X_{\gr{i_1}   },\dots,X_{\gr{i_N}  }}$ have the same distribution.
\end{Definition}

It will be convenient to represent strictly stationary random 
field via dynamical systems.  
 Let $\pr{X_{\gri}}_{\gri\in \Z^d}$ be a strictly stationary 
 random field on a probability space $\pr{\Omega,\Fca,\PP}$. 
 Then there exists a probability space 
 $\pr{\Omega',\Fca',\PP'}$, $f\colon \Omega\to \R$
 and maps $T_q\colon \Omega'\to 
 \Omega'$ which are invertible, bi-measurable and measure 
 preserving and commuting
 such that $\pr{X_{\gri}}_{\gri\in \Z^d}$ has 
 the same distribution as  $\pr{f\circ T^{\gri}}_{\gri\in \Z^d}$, where $T^{\gri}=T_1^{i_1}\circ \dots\circ T_d^{i_d}$.

Since the behaviour of supremum of the weighted partial 
sums depends only on the law of the random field, we 
will assume without loss of generality that the involved 
stationary random field is of the form
$\pr{f\circ T^{\gri}}_{\gri\in\Z^d}$ and use the notations 
$U^{\gri}\pr{f}\pr{\omega}=f\pr{T^{\gri}\omega}$ and 
\begin{equation}
 S_{\grn}\pr{f}=\sum_{\gr{0}\imd\gri\imd\grn-\gr{1}}
 U^{\gri}\pr{f}.
\end{equation}

\section{Orthomartingale case}

In order to use martingale methods for the law of 
iterated logarithms, we need to introduce the 
concept of orthomartingales, which can be viewed as 
a generalization of martingales. Orthomartingales are also 
an adapted tool in order to treat the summations on rectangles. 
However, the theory of orthomartingales bumps into obstacles 
which are technical. Indeed, the extension of the notion 
of stopping time is not clear. Moreover, in most of the 
exponential inequality for martingales \cite{MR3734020,MR3311214,MR2462551}, 
the sum of square of increments and conditional variances plays a key role. A multidimensional 
equivalent does not seem obvious.

\subsection{Definition of orthomartingales}

We start by defining the meaning of filtrations in the multi-dimensional setting.

\begin{Definition}
 We call the collection of sub-$\sigma$-algebras $
 \pr{\Fca_{\gr{i}}}_{\gr{i}\in \Z^d}$ of $\Fca$ a 
 \emph{filtration} if for all $\gri,\grj\in\Z^d$ such 
 that $\gri\imd \grj$, the inclusion $\Fca_{\gri}\subset 
 \Fca_{\grj}$ holds.
\end{Definition}

In order to have filtrations compatible with the map $T$, 
we will consider filtrations of the form $\Fca_{\gri}:=
T^{-\gri}\Fca_{\gr{0}}$. These are indeed filtrations 
provided that $T_q\Fca_{\gr{0}}\subset \Fca_{\gr{0}}$
holds for all $q\in [d]$. 

We will also impose commutativity of the involved filtrations, 
that is, for each integrable random variable $Y$, the following 
equalities should hold for all $\gri$ and $\grj\in\Z^d$:
\begin{equation}
 \E{\E{Y\mid \Fca_{\gri}  }\mid \Fca_{\grj}}= 
 \E{\E{Y\mid \Fca_{\grj}  }\mid \Fca_{\gri}}
 =  \E{Y\mid \Fca_{\min\ens{\gri,\grj}  }  },
\end{equation}
where $\min\ens{\gri,\grj}$ is the coordinatewise minimum, that is, 
$\min\ens{\gri,\grj}=\pr{\min\ens{i_q,j_q}}_{q=1}^d$.

\begin{Definition}
 Let $\pr{T^{-\gri}\Fca_{\gr0}}_{\gri\in\Z^d}$ be a commuting 
 filtration. We say that $\pr{m\circ T^{\gri}}_{\gri\in\Z^d}$ 
 is an \emph{orthomartingale difference random field} if 
 the function $m$ is integrable and $\Fca_{\gr{0}}$-measurable 
 such that $\E{m\mid T_q\Fca_{\gr{0}}}
 =0$ for all $q\in [d]$.
\end{Definition}

Strictly stationary orthomartingale difference random fields 
are a convenient class of random fields to deal with, especially 
from the point of view of limit theorems. If 
$\pr{m\circ T^{\gri}}_{\gri\in\Z^d}$ is an orthomartingale difference 
random field and one of the maps $T^{\gr{e_q}}$ is ergodic, 
then $\pr{S_{\grn}\pr{m}/\abs{\grn}}_{\grn\smd \gr{1}}$ converges 
to a normal distribution as $\max\grn$ goes to infinity 
(see \cite{MR3427925}). Under these conditions, a functional 
central limit theorem has also been established in Theorem~1 of 
\cite{MR3504508}. 

It turns out that a central limit theorem 
still holds without the assumption of ergodicity of one of 
the marginal transformations $T^{\gr{e_q}}$ (see 
Theorem~1 in \cite{VOLNY2018}). However, it seems that 
there is no result regarding the law of the iterated logarithms 
for orthomartingale difference random fields.

\subsection{Definition of the maximal function}

Consider one of the most simple example of orthomartingale 
difference random fields in dimension two defined in the 
following way: let $\Omega:=\Omega_1\times\Omega_2$, where 
$\pr{\Omega_1,\Aca_1,\mu_1,T_1}$ and 
$\pr{\Omega_2,\Aca_2,\mu_2,T_2}$
are dynamical systems, where $\Aca_1$ and $\Aca_2$ are 
generated respectively by 
$e_1\circ T_1^{i_1}$, $i_1\in\Z$ and $e_2\circ T_2^{i_2}$, 
$i_2\in\Z$ and $e_1$, $e_2$ are bounded centered 
functions such that the sequences
$\pr{e_1\circ T_1^{i_1}}_{i_1\in\Z}$ and 
$\pr{e_2\circ T_2^{i_2}}_{i_2\in\Z}$ are both i.i.d.
Define $X_{\pr{i_1,i_2}}:=e_1\circ T_1^{i_1}\cdot e_2\circ T_2^{i_2}$ 
and let $\Fca_{\pr{i_1,i_2}}:=\sigma\ens{X_{\pr{j_1,j_2}},j_1\leq i_1
\mbox{ and }j_2\leq i_2}$. Then $\Fca_{\pr{i_1,i_2}}=
T_1^{-i_1}T_2^{-i_2}\Fca_{0,0}$ and 
$\pr{\Fca_{\pr{i_1,i_2}}}_{i_1,i_2\in\Z}$ is a commuting filtration. 
Moreover, $\pr{X_{\pr{0,0}}\circ T^{\pr{i_1,i_2}}}_{i_1,i_2\in\Z}$ 
is an orthomartingale difference random field and $X_{\pr{0,0}}$ is 
bounded. Observe that for all $n_1,n_2\geq 1$, the 
following inequality holds 
\begin{equation}
 \frac{1}{\sqrt{n_1n_2L L\pr{n_1n_2   }      }}
 \abs{S_{\pr{n_1,n_2}}}
 = \frac{1}{\sqrt{n_1L L\pr{n_1n_2    }    }}
 \abs{\sum_{i_1=0}^{n_1-1}e_1\circ T_1^{e_1}} 
 \frac{1}{\sqrt{n_2}}\abs{\sum_{i_2=0}^{n_2-1}
 e_2\circ T_2^{e_2}}, 
\end{equation}
which can be rewritten as 
\begin{multline}
  \frac{1}{\sqrt{n_1n_2 LL\pr{n_1n_2   }      }}
 \abs{S_{\pr{n_1,n_2}}}\\
 = \frac{1}{\sqrt{n_1L L\pr{n_1   }      }}
 \abs{\sum_{i_1=0}^{n_1-1}e_1\circ T_1^{e_1}} 
 \frac{\sqrt{L L\pr{n_1    }    }}{\sqrt{
 L L\pr{n_1n_2      }    }}
 \frac 1{\sqrt{n_2}}\abs{\sum_{i_2=0}^{n_2-1}
 e_2\circ T_2^{e_2}}.
\end{multline}
Consequently, for any fixed $n_2\geq 1$, it holds, from the 
classical law of the iterated logarithms and the fact that 
\begin{equation}
 \frac{\sqrt{L L\pr{n_1      }    }}{\sqrt{ L L\pr{n_1n_2   }  }    }\to 1,
\end{equation}
that 
\begin{equation}
 \sup_{n_1\geq 1}\frac{1}{\sqrt{n_1n_2L L\pr{n_1n_2     }    }}
 \abs{S_{\pr{n_1,n_2}}}
 \geq  \frac 1{\sqrt 2}\norm{e_1}_2\frac 1{\sqrt{n_2}}\abs{\sum_{i_2=0}^{n_2-1}
 e_2\circ T_2^{e_2}}.
\end{equation}
Hence the same maximal function as in the Bernoulli case (see \cite{giraudo2019bounded}) would be 
almost surely infinite. This lead to an alternative definition, 
namely,  
\begin{equation}
M\pr{f}:=\sup_{\gr{n}\in\N^d}\frac{\abs{S_{\gr{n}}\pr{f}}}{
\abs{\gr{n}}^{1/2}\pr{\prod_{i=1}^d
L L\pr{n_i   }}^{1/2}}.
\end{equation}

This definition is coherent with the previous example of 
orthomartingale and its generalization to the dimension $d$. 
In this case, $M\pr{X_{\pr{0,0}}}$ is simply the product of the 
$1$-dimensional maximal function associated to bounded 
i.i.d. sequences. Hence $M\pr{X_{\pr{0,0}}}$ is almost surely finite. 

\subsection{Result}

Before we present the result, we need to introduce spaces 
of random variables satisfying an integrability condition, which is more 
restrictive than having a finite moment of order two but less restrictive
than a finite moment of order $2+\delta$ for a positive $\delta$. 
We define  for $p>1$ and $r\geq 0$, the function $\varphi_{p,r}\colon 
[0,+\infty)$ by $\varphi_{p,r}\pr{x}:=x^p\pr{1+\log\pr{1+x}    }^r$ and 
 denote by 
$\el_{p,r}$ the Orlicz space associated to this function. We define the 
norm $\norm{\cdot}_{p,r}$ of an element $X$ of $\el_{p,r}$ by 
\begin{equation}
 \norm{X}_{p,r}:=\inf\ens{\lambda>0\mid 
 \E{\varphi_{p,r}\pr{\frac X\lambda}  }\leq 1}.
\end{equation}

It turns out that for a stationary orthomartingale difference 
sequence, the maximal function is almost surely finite 
provided that $m$ belongs to $\el_{2,2\pr{d-1}}$. The 
next result gives also a control for the moments of the 
maximal function.

\begin{Theorem}\label{thm:LIL_orthomartingales}
Let $d\geq 1$ be an integer. For all $1\leq p<2$, there exists a constant 
$C_{p,d}$ depending only on $p$ and $d$ such that for 
all  strictly stationary orthomartingale difference random field
 $\pr{m\circ T^{\gr{i}}}_{\gr{i}\in \Z^d}$, the following inequality holds:
\begin{equation}\label{eq:norme_p_fct_max_orthomartingale}
\norm{M\pr{m}}_p\leq C_{p,d}\norm{m}_{2,2\pr{d-1}}.
\end{equation}
Moreover, for all $r\geq 0$, 
\begin{equation}\label{eq:norme_2_fct_max_orthomartingale}
\norm{M\pr{m}}_{2,r}\leq C_{p,d,r}\norm{m}_{2,r+2d}.
\end{equation}
\end{Theorem}

\begin{Remark}
Theorem~2.3 of \cite{MR3322323} gives a similar result as Theorem~\ref{thm:LIL_orthomartingales} for stationary martingale difference sequences, which corresponds to $d=1$ in Theorem~\ref{thm:LIL_orthomartingales}.
The case of Banach-valued martingales was addressed in the former paper but 
not the case of orthomartingales.

Moreover, \eqref{eq:norme_2_fct_max_orthomartingale} gives 
 $\norm{M\pr{m}}_{2,r}\leq C_{p,d,r}\norm{m}_{2,r+2}$ and in particular a control on the $\mathbb L_{2,r}$-norm 
 of $M\pr{m}$. 
\end{Remark}

\begin{Remark}
The condition $m\in \el_{2,2\pr{d-1}}$ is sufficient for the bounded law 
of the iterated logarithms. However, we are not able to determine whether 
the parameter $2\pr{d-1}$ is optimal.
\end{Remark}

\section{Projective conditions}

Given a probability space $\pr{\Omega,\Fca,\PP}$ endowed 
with a measure preserving action $T$ and a commuting 
filtration $\pr{T^{-\gri}\Fca_{\gr{0}}}_{\gri\in\Z^d}$, a
projective condition is a requirement on a function 
$f\colon\Omega\to \R$ involving the functions 
$\E{f\circ T^{\gri}\mid\Fca_{\gr{0}}}$, $\gri\in\Z^d$. In the orthomartingale case, 
the function $\E{f\circ T^{\gri}\mid\Fca_{\gr{0}}}$ is identically equal to zero if 
$\gri\smd 0$ and $\gri\neq 0$. Therefore, projective conditions can be intuitively seen 
as a measure of the distance with respect to the martingale case.

\subsection{Hannan-type condition}

If $\pr{\Fca_{\gr{i}}}_{\gr{i}\in \Z^d}$ is a commuting filtration and $J\subset [d]$, we denote 
by $\Fca_{\infty\gr{1}_I+\gr{i}}$ the $\sigma$-algebra generated by the union of $\Fca_{\gr{j}}$ 
where $\gr{j}$ runs over all the elements of $\Z^d$ such that $j_q \leq i_q$ for 
all $q\in [d]\setminus I$. For example, if $d=2$, then $\Fca_{\infty\gr{1}_{\ens{1}}+\pr{i_1,i_2}}$ 
is the $\sigma$-algebra generated by $\bigcup_{j_1\in \Z}\Fca_{\pr{j_1,i_2}}$ and $\Fca_{\infty\gr{1}_{\ens{2}}+\pr{i_1,i_2}}$ 
is the $\sigma$-algebra generated by $\bigcup_{j_2\in \Z}\Fca_{\pr{i_1,j_2}}$.
Let $\pr{U^{\gri}f}_{\gri\in\Z^d}$ be a strictly stationary random field. 

Assume first that $d=1$, $T\colon \Omega\to \Omega$ is a bijective 
 bimeasurable measure preserving map and $\f_0$ is a sub-$\sigma$-algebra 
 such that $T\f_0\subset\f_0$. Assume that $f\colon \Omega\to \R$ 
 is measurable with respect to the $\sigma$-algebra generated by 
 $\bigcup_{k\in \Z}T^k\f_0$ and such that $\E{f\mid \bigcap_{k\in \Z}
 T^k\f_0}=0$. Let us consider the condition
 \begin{equation}\label{eq:Hannan_dim1}
  \sum_{i\in \Z}\norm{\E{f\circ T^i\mid \f_0}
  -\E{f\circ T^i\mid T\f_0}}_2<+\infty.
 \end{equation}

 The generalization of condition \eqref{eq:Hannan_dim1} 
 to random fields has been considered by Voln\'y and Wang.
  Let us recall the notations and results of \cite{MR3264437}.
 The projection operators with respect to a commuting filtration 
$\left(\mathcal F_{\gri}\right)_{\gri\in \Z^d}$ are 
defined by 
  \begin{equation}\label{eq:definition_projectors}
 \pi_{\grj}:=\prod_{q=1}^d\pi_{j_q}^{(q)},\quad \grj\in \Z^d,
\end{equation}
where for $\ell\in \Z$, $\pi_\ell^{(q)}\colon \mathbb L^1(\mathcal F)\to \mathbb L^1(\mathcal F)$ 
is defined for $f\in\mathbb L^1$ by 
\begin{equation}
 \pi_\ell^{(q)}(f)=\mathbb E_\ell^{(q)}\left[f\right]
 -\mathbb E_{\ell-1}^{(q)}\left[f\right]
\end{equation}
and 
\begin{equation}
 \mathbb E_\ell^{(q)}\left[f\right]=\mathbb E\left[f\mid 
 \bigvee_{\mathclap{\substack{\gri\in \Z^d \\
 i_q\leqslant \ell}}}\mathcal F_{\gri}\right], q\in [d],\ell\in \Z.
\end{equation} The natural extension of \eqref{eq:Hannan_dim1} 
to the dimension $d$ case is 
\begin{equation}\label{eq:Hannan_dim_d_pour_le_WIP}
 \sum_{\grj\in\Z^d}\norm{ \pi_{\grj}\pr{f}}_{2}<+\infty.
\end{equation} Under \eqref{eq:Hannan_dim_d_pour_le_WIP}, the functional 
central limit holds (Theorem~5.1 in \cite{MR3264437} and 
Theorem~8 in \cite{MR3504508} and its quenched version 
\cite{2018arXiv180908686Z}). 
Theorefore, it is reasonnable 
to look for a condition in this spirit for the bounded law 
of the iterated logarithms. The obtained result is as 
follows.

  \begin{Theorem} \label{thm:Hannan_con}
   Let $\pr{\f_{\gr{i}}}_{\gr{i}\in\Z^d}:=
   \pr{T^{-\gri}\f_{\gr{0}}}_{\gri\in \Z^d}$ 
   be a commuting filtration. 
   Let $f$ be a function such that for each $q\in [d]$, $\E{ 
   f\mid T_q^\ell\f_{\gr{0}} }\to 0$ as $\ell\to +\infty$ and measurable with respect to the 
   $\sigma$-algebra generated by $\bigcup_{\gri\in\Z^d}T^{\gri}
   \f_{\gr{0}}$. Then for all $1<p<2$, 
   \begin{equation}
   \norm{M\pr{f}}_p\leq C_{p,d} 
   \sum_{\grj\in\Z^d}\norm{ \pi_{\grj}\pr{f}}_{2,2\pr{d-1}}. 
   \end{equation}
  \end{Theorem}

\subsection{Maxwell and Woodroofe type condition}  \label{sec:MW}

In order to extend the results obtained for 
orthomartingales to a larger class of strictly stationary 
random fields, we need an 
extension of the following almost sure maximal inequality 
(Proposition~4.1 in \cite{MR3650410}).

\begin{Proposition}
Let $\pr{\Omega,\Fca,\mu,T}$ be a dynamical system and let 
$\Fca_0$ be a sub-$\sigma$-algebra of $\Fca$ such that 
$T\Fca_0\subset \Fca_0$. Denote $\proj{Y}{j}:=\E{Y\mid T^{-j}\Fca_0}$.
Then for all integer $n\geq 0$ and all $\Fca_0$-measurable function $f$, the 
following inequality holds almost surely:
\begin{multline}\label{eq:almost_sure_inequality_dim_1}
\max_{1\leq i\leq 2^n}\abs{\sum_{j=0}^{i-1}f\circ T^j}
\leq \max_{1\leq i\leq 2^n}\abs{\sum_{\ell =0}^{i-1}\pr{f-\proj{f}{-1}}\circ T^\ell}
+\sum_{k=0}^{n-1}\max_{1\leq i\leq 2^{n-k-1}}
\abs{\sum_{\ell=0}^{i-1}d_k\circ T^{2^{k+1}\ell}}\\
+\abs{u_n}+\sum_{k=0}^{n-1}\max_{1\leq \ell\leq 2^{n-k-1}-1}
\abs{u_k}\circ T^{2^{k+1}\ell},
\end{multline}
where 
\begin{equation}
u_k=\proj{\sum_{j=0}^{2^k-1}f\circ T^j   }{-2^k},
\end{equation}
\begin{equation}
d_k=u_k+u_k\circ T^{2^k}-u_{k+1}.
\end{equation}
\end{Proposition}
We can observe that for all fixed $k$, the collection of random variables 
$\pr{d_k\circ T^{2^{k+1}\ell}}_{\ell\geq 0}$ is 
a martingale difference sequence, while for each fixed $k$, the 
contribution of $u_k$ is analoguous as that of a coboundary. 

The goal of the next proposition is to extend the previous almost sure 
inequality to the dimension $d$. It turns out that an analogous 
inequality can be established, where the decomposition while involve 
orthomartingale difference random fields in some coordinates and 
coboundary in the other one. In order to formulate this, we need the 
following notation. If $T$ is a measure preserving $\Z^d$- action on 
$\pr{\Omega,\Fca,\mu}$, $\gri\in\N^d$, $I\subset [d]$ and $h\colon\Omega\to \R$
, we define 
\begin{equation}\label{eq:sommes_partielles_dans_une_direction}
S_{\gri}^I\pr{T,h}:=\sum_{ 
\substack{0\leq j_q\leq i_q-1\\ q\in I}
   }
   h\circ T^{\sum_{q'\in I}j_{q'}\gr{e_{q'}}+
   \sum_{q''\in [d]\setminus I}i_{q''}\gr{e_{q''}}     }.
\end{equation}
In other words, the summation is done on the coordinates of the set $I$ and the 
coordinates of $[d]\setminus I$ are equal to the corresponding ones of $\gri$. In particular, 
for $I=[d]$, this is nothing but the classical partial sums.
We will need also the following notations: for $\grk\in\Z^d$, we denote by 
$Z\pr{\grk}$ the set of the elements $q\in [d]$ such that $k_q=0$. Moreover, 
given a commuting filtration $\pr{T^{-\gri}\Fca_{\gr{0}}}_{\gri\in\Z^d}$ and 
an integrable random variable $X$, we define the operator $\proj{X}{\gri}$ 
by 
\begin{equation}
\proj{X}{\gri}:=\E{X\mid T^{-\gri}\Fca_{\gr{0}}}.
\end{equation}

We are now in position to state the following almost sure inequality for 
stationary random fields.

\begin{Proposition}\label{prop:inegalite_presque_sure_dim_d}
Let $T$ be a measure preserving $\Z^d$-action on a probability space 
$\pr{\Omega,\Fca,\mu}$. Let $\Fca_{\gr{0}}\subset\Fca$ 
be a sub-$\sigma$-algebra such that $T^{\gr{e_q}}\Fca_{\gr{0}}\subset 
\Fca_{\gr{0}}$ for all $q\in [d]$ and the filtration 
$\pr{T^{-\gri}\Fca_{\gr{0}}}_{\gri\in\Z^d}$ is commuting. For each
$\Fca_{\gr{0}}$-measurable function $f$, the following inequality 
takes place almost surely:
\begin{equation}\label{eq:almost_sure_inequality_dim_d}
\max_{\gr{1}\imd \gr{i}\imd\gr{2^n}}\abs{S_{\gri}\pr{f}}   \leq  
\sum_{\gr{0}\imd\gr{k}\imd\gr{n}}
\sum_{I\subset [d]}
\max_{\gr{1}-\gr{1}_{I}\imd\gri\imd 
\gr{2^{n-k}}}\abs{S_{\gri}^I\pr{T^{\gr{2^k}   },d_{\grk,I}   }  },
\end{equation}
where 
\begin{equation}\label{eq:definition_de_dkI}
d_{\grk,I}:= 
\sum_{I''\subset    I\setminus Z\pr{\grk}}
\sum_{I'\subset  I\cap Z\pr{\grk} }
\pr{-1}^{\abs{I'}+\abs{I''}}
\proj{S_{2^{\grk+\gr{1_{Z\pr{k}}}-\gr{1_I}}}\pr{f}  }{-2^{\grk -\gr{1_{I''}}}-\gr{1_{I'}}}.
\end{equation}
\end{Proposition}
 Observe that for each $I\subset [d]$ and for all $\grk$ such that 
 $\gr{0}\imd\grk \imd \grn$, the random field 
 $\pr{d_{\grk,I}\circ T^{2^{\grk}\gri_{ I }   }}_{\gri_I\in \Z^{\abs{I}}  }$ 
 is an orthomartingale difference random field. 
 In particular,  taking the $\el^2$-norm (resp. $\el^p$)
   on both sides of the inequality allows us to 
 recover Proposition~2.1 in \cite{MR3869881} (resp. Proposition~7.1 of 
 \cite{MR3222815}) in the adapted case. 
 
In order to have a better understanding of the terms involved in the right hand side of 
\eqref{eq:almost_sure_inequality_dim_d}, we will write this inequality in dimension $2$. 
This becomes 
\begin{multline}\label{eq:almost_sure_inequality_dim_2}
\max_{\substack{   1\leq i_1\leq 2^{n_1}\\
 1\leq i_2\leq 2^{n_2}}}\abs{S_{i_1,i_2}\pr{f}}   \leq  
\sum_{k_1=0}^{n_1}\sum_{k_2=0}^{n_2}
\max_{ \substack{   1\leq i_1\leq 2^{n_1-k_1}\\  1\leq i_2\leq 2^{n_2-k_2}    }
 }\abs{S_{i_1,i_2}\pr{T^{2^{k_1},2^{k_2}  },d_{k_1,k_2,[2]}   }  }\\ 
 +\sum_{k_1=0}^{n_1}\sum_{k_2=0}^{n_2}
\max_{ \substack{   1\leq i_1\leq 2^{n_1-k_1}\\  0\leq i_2\leq 2^{n_2-k_2}    }
 }\abs{S_{i_1,1}\pr{T^{2^{k_1},0  },d_{k_1,k_2,\ens{1}}   }\circ T^{0,2^{k_2}i_2}  } \\
  +\sum_{k_1=0}^{n_1}\sum_{k_2=0}^{n_2}
\max_{ \substack{   0\leq i_1\leq 2^{n_1-k_1}\\  1\leq i_2\leq 2^{n_2-k_2}    }
 }\abs{S_{1,i_2}\pr{T^{0,2^{k_1}  },d_{k_1,k_2,\ens{2}}   }\circ T^{2^{k_1}i_1,0}  }\\
 + \sum_{k_1=0}^{n_1}\sum_{k_2=0}^{n_2}
\max_{ \substack{   0\leq i_1\leq 2^{n_1-k_1}\\  0\leq i_2\leq 2^{n_2-k_2}    }
 }\abs{  d_{k_1,k_2,\emptyset}   }\circ T^{2^{k_1}i_1,2^{k_2}i_2}    ,
\end{multline}
where for $k_1,k_2\geq 1$, 
\begin{equation}
d_{0,0,[2]}=f-\proj{f}{0,-1}-\proj{f}{-1,0}+\proj{f}{-1,-1},
\end{equation}
\begin{multline}
d_{k_1,0,[2]}= \proj{S_{2^{k_1-1},1} \pr{f}}{ -2^{k_1-1},0     } -
\proj{S_{2^{k_1-1},1} \pr{f}}{ -2^{k_1},0     } \\-
\proj{S_{2^{k_1-1},1} \pr{f}}{ -2^{k_1},-1     }+
\proj{S_{2^{k_1-1},1} \pr{f}}{ -2^{k_1-1},-1     }    ,
\end{multline} 
 \begin{multline}
d_{0,k_2,[2]}= \proj{S_{1,2^{k_2-1}} \pr{f}}{ 0,-2^{k_2-1}     } -
\proj{S_{1,2^{k_2-1}} \pr{f}}{ 0,-2^{k_2 }      } \\-
\proj{S_{1,2^{k_2-1}} \pr{f}}{ -1,-2^{k_2-1}      }+
\proj{S_{1,2^{k_2-1}} \pr{f}}{-1,-2^{k_2}     }    ,
\end{multline} 
 \begin{multline}
d_{k_1,k_2,[2]}= \proj{S_{2^{k_1-1},2^{k_2-1}} \pr{f}}{ 2^{k_1-1},-2^{k_2-1}     } -
\proj{S_{1,2^{k_2-1}} \pr{f}}{ -2^{k_1-1},-2^{k_2 }      } \\-
\proj{S_{1,2^{k_2-1}} \pr{f}}{ -2^{k_1},-2^{k_2-1}      }+
\proj{S_{1,2^{k_2-1}} \pr{f}}{-2^{k_1},-2^{k_2}     }    ,
\end{multline} 
 \begin{equation}
 d_{0,0,\ens{1}}= 
 \proj{S_{2^{k_1},2^{k_2+1}  }\pr{f}  }{-2^{k_1},-2^{k_2}}
 -\proj{S_{2^{k_1},2^{k_2+1}  }\pr{f}  }{-2^{k_1}-1,-2^{k_2}},
 \end{equation}
  \begin{equation}
 d_{k_1,0,\ens{1}}=
 \proj{S_{2^{k_1},2^{k_2}  }\pr{f}  }{-2^{k_1},-2^{k_2}}
 -\proj{S_{2^{k_1},2^{k_2}  }\pr{f}  }{-2^{k_1}-1,-2^{k_2}},
 \end{equation}
  \begin{equation}
 d_{0,k_2,\ens{1}}= \proj{S_{2^{k_1-1},2^{k_2+1}  }\pr{f}  }{-2^{k_1},-2^{k_2}}
 -\proj{S_{2^{k_1-1},2^{k_2+1}  }\pr{f}  }{-2^{k_1}-1,-2^{k_2}},
 \end{equation}
  \begin{equation}
 d_{k_1,k_2,\ens{1}}= \proj{S_{2^{k_1},2^{k_2}}\pr{f}
 }{-2^{-k_1},2^{-k_2}}-\proj{S_{2^{k_1},2^{k_2}}\pr{f}
 }{-2^{-k_1-1},2^{-k_2}},
 \end{equation}
 a similar expression for $ d_{k_1,k_2,\ens{2}}$ by 
 switching the roles of $T_1$ and $T_2$  and 
\begin{equation}
  d_{0,0,\emptyset}= \proj{S_{2^{k_1+1},2^{k_2+1}}\pr{f}  }{-2^{k_1},-2^{k_2}},
\end{equation}
\begin{equation}
  d_{k_1,0,\emptyset}= \proj{S_{2^{k_1},2^{k_2+1}}\pr{f}  }{-2^{k_1},-2^{k_2}} ,
\end{equation}
\begin{equation}
  d_{0,k_2,\emptyset}=  \proj{S_{2^{k_1+1},2^{k_2}}\pr{f}  }{-2^{k_1},-2^{k_2}},
\end{equation}
\begin{equation}
  d_{k_1,k_2,\emptyset}= \proj{S_{2^{k_1},2^{k_2}}\pr{f}  }{-2^{k_1},-2^{k_2}} .
\end{equation}

 We are now in position to state a result for the law of the iterated logarithms 
 under a condition analogue to the Maxwell and Woodroofe condition, that is, 
 involving the norm in some space of $\E{S_{\grn}\pr{f} \mid \Fca_{\gr{0}  }     }$.

\begin{Theorem}\label{thm:MW_condition}
Let $T$ be a $\Z^d$-measure preserving action on a probability 
space $\pr{\Omega,\Fca,\mu}$. Let $\Fca_{\gr{0}}$ be a 
sub-$\sigma$-algebra of $\Fca$ such that $\pr{T^{-\gri}\Fca_{\gr{0}}}_{\gri
\in\Z^d}$ is a commuting filtration. Let $1<p<2$. There exists a constant 
$c_{p,d}$ such that for all $\Fca_{\gr{0}}$-measurable function $f\colon \Omega\to \R$, 
the following inequality holds:
 \begin{equation}\label{eq:inequalite_LLI_MW}
  \norm{M\pr{f}}_p   \leq c_{p,d} \sum_{\grn\smd\gr{1}}
 \frac 1{\abs{\grn}^{3/2} }
 \norm{ \E{S_{\grn}\pr{f} \mid \Fca_{\gr{0}  }     }}_{2,2\pr{d-1}}.
 \end{equation}
\end{Theorem}

\subsection{Application}

The previous conditions can be checked for linear processes whose innovations are 
orthomartingale difference random fields. 

\begin{Corollary}\label{cor:application_champs_lineaires_cond_proj}
Let $\pr{m\circ T^{\gri}}_{\gri\in \Z^d}$ be a strictly stationary orthomartingale 
difference random field with $m\in \el_{2,2\pr{d-1}}$, let $\pr{a_{\gri}}_{\gri\in \Z^s}\in \ell^2\pr{\Z^d}$ 
and let $\pr{f\circ T^{\gri}}_{\gri\in \Z^d}$ be the causal linear random field defined by 
\begin{equation}
f\circ T^{\gri}=\sum_{\grj \smd\gr{0}  }a_{\grj}m\circ T^{\grj-\gri}.
\end{equation}
Then for all $1<p<2$, the following inequalities take place:
\begin{equation}\label{eq:Hannan_processus_lineaires}
\norm{M\pr{f}}_p\leq C_{p,d}\sum_{\gri\smd \gr{0}}\abs{   a_{\gri}  }\norm{m}_{2,2\pr{d-1}};
\end{equation}
\begin{equation}\label{eq:MW_processus_lineaires}
\norm{M\pr{f}}_p\leq C_{p,d}\sum_{\grn\smd \gr{1}} 
\frac 1{\abs{\grn}^{3/2}}
 \pr{ \sum_{\gr{\ell}\smd\gr{0}}\pr{\sum_{\gr{0}\imd \gri\imd \grn-\gr{1}}
a_{ \gri + \gr{\ell}  }}^2}^{1/2}\norm{m}_{2,2\pr{d-1}},
\end{equation}
where $C_{p,d}$ depends only on $p$ and $d$.
\end{Corollary}

\begin{Remark}
In \cite{giraudo2019bounded}, linear processes were also investigated but with the 
assumption that the innovations are i.i.d.. In this case, the normalization in the definition of 
maximal function is weaker.
\end{Remark}     
 
 One of the points of considering orthomartingale innovations is the decomposition of a stationary process 
 as a sum of linear process. More precisely, let $\pr{T^{-\gri}\Fca_{\gr{0}}}_{\gri\in\Z^d}$ be a commuting
 filtration. Define the subspaces  
 \begin{equation}
 V_d :=\ens{f\in\mathbb L^1, f\mbox{ is }\Fca_{\gr{0}}-\mbox{measurable and for all }q\in [d], \E{f\mid T^{\gr{e_q}}\Fca_{\gr{0}}   }=0}
 \end{equation}
\begin{equation}\label{eq:dfn_espace_W_d}
W_d=V_d\cap \mathbb L_{2,2\pr{d-1}}.
\end{equation}

\begin{Corollary}\label{cor:application_champs_multi_lineaires_cond_proj}
Assume that there exists a sequence $\pr{e_k}_{k\geq 1}$ of elements of $W_d$ such that 
each element $f$ of $W_d$ can be writen as $\sum_{k=1}^{+\infty}c_ke_k$, where the limit 
is taken with respect to the $\mathbb L_{2,2\pr{d-1}}$-norm and $\norm{e_k}_{2,2\pr{d-1}}\leq 1$. Let $f$ be
 an $\f_{\gr{0}}$-measurable function such that for each $q\in [d]$, $\E{ 
   f\mid T_q^\ell\f_{\gr{0}} }\to 0$ as $\ell\to +\infty$. Then $f$ admits the representation 
   \begin{equation}\label{eq:representation_de_f}
   f=\sum_{\grj \smd \gr{0}} a_{k,\grj}\pr{f} U^{-\grj }e_k
   \end{equation}
   and for all $1<p<2$, the following inequalities holds:
   \begin{equation}\label{eq:Hannan_processus_multi_lineaires}
\norm{M\pr{f}}_p\leq C_{p,d}\sum_{k\geq 1}\sum_{\gri\smd \gr{0}}\abs{   a_{k,\gri}\pr{f}  } ;
\end{equation}
   \begin{equation}\label{eq:MW_processus_multi_lineaires}
\norm{M\pr{f}}_p\leq C_{p,d}\sum_{k\geq 1} \sum_{\grn\smd \gr{1}} 
\frac 1{\abs{\grn}^{3/2}}
 \pr{ \sum_{\gr{\ell}\smd\gr{0}}\pr{\sum_{\gr{0}\imd \gri\imd \grn-\gr{1}}
a_{ k,\gri + \gr{\ell}  }\pr{f}}^2}^{1/2}.
\end{equation}
\end{Corollary}

\section{Proofs}

\subsection{Tools for the proofs}

\subsubsection{Global ideas of proofs}

Let us explain the main steps in the proofs of the results. 

Let us first focus on 
orthomartingale differences. The maximal 
function is defined as a supremum over all the $\grn \in \N^d$.
However, due to the lack of exponential inequalities for the 
maximal of partial sums on rectangles, we will instead work 
with other maximal functions, where the supremum is 
restricted to the elements of $\N^d$ whose components 
are powers of two. The martingale property helps 
to show that the moments of the former maximal function 
are bounded up to a constant by those of the later. 

We then have to control the deviation probability 
of the sum on a rectangle. It is convenient to 
control the latter probability intersected with the 
event where the sum (in one direction) of squares and conditional variances
of the random field is bounded 
by some $y$. The contribution of this term can be controlled 
by an application of the maximal ergodic theorem and we are left  
to control moment of maximal functions in lower dimension.
 Then we use an induction argument. 
 
For results concerning projective conditions, they are 
consequences of the result for orthomartingales after 
an appropriated decomposition of the involved random 
field.

\subsubsection{Weak $\mathbb L^p$-spaces}

The results of the paper involve all a control of 
the $\el^p$ norm of a maximal function. However, 
it will sometimes be more convenient to work directly 
with tails. To this aim, we will consider weak
 $\el^p$-spaces.

\begin{Definition}
Let $p>1$. The weak $\el^p$-space, denoted by $\el^{p,w}$, is the 
space of random variables $X$ such that $\sup_{t>0}t^p\PP\ens{\abs X>t}$ 
is finite. 
\end{Definition}

These spaces can be endowed with a norm.

\begin{Lemma}\label{lem:lp_faibles}
Let $1<p\leq 2$. Define the following norm on $\el^{p,w}$
\begin{equation}
 \norm{X}_{p,w}:=\sup_{A\in\mathcal F, 
 \PP\pr{A}>0}  \PP\pr{A}^{1/p-1}\E{\abs{X}\mathbf 1_A}.
\end{equation}

For all random variable $X\in \el^{p,w}$, the following inequalities hold:
\begin{equation}\label{eq:ineg_weak_Lp}
c_p\norm{X}_{p,w}\leq \pr{\sup_{t>0}t^p\PP\ens{\abs X>t}  }^{1/p}
\leq C_p\norm{X}_{p,w}\leq C_p\norm{X}_{p},
\end{equation}
where the positive constants $c_p$ and $C_p$ depend only on $p$.
\end{Lemma}

Let us give some ideas of proof. For the first inequality, express 
$\E{\abs{X}\mathbf 1_A}$ as an integral of the tail of the random variable 
$\abs{X}\mathbf 1_A$ and bound this tail by $\min\ens{t^{-p}\sup_{t>0}t^p\PP\ens{\abs X>t}  ,\PP\pr{A}}$. For the second inequality, bound $t^p\PP\ens{\abs X>t}$ by 
$t^{p-1}\E{\abs{X}\mathbf{1}\ens{\abs X>t}}$ and use the definition of 
$\norm{\cdot}_{p,w}$ to get that 
\begin{equation}
 t^p\PP\ens{\abs X>t}\leq \norm{X}_{p,w}t^{p-1}\PP\ens{\abs{X}>t}^{1-1/p}.
\end{equation}
Hence for each $t>0$, $t\pr{\PP\ens{\abs X>t}}^{1/p}\leq \norm{X}_{p,w}$.  
Finally, the last inequality in \eqref{eq:ineg_weak_Lp} follows from Hölder's inequality.

\subsubsection{Deviation inequalities}

The following deviation inequality is a consequence of Theorem~2.1 in 
\cite{MR2462551}.
 \begin{Proposition}\label{prop:inegalite_deviation_martingales}
  Let $\pr{d_j}_{j\geq 1}$ be a sequence of
square integrable martingale differences with 
  respect to the filtration $\pr{\Fca_j}_{j\geq 0}$. Then for all positive numbers $x$ and $y$, 
  the following inequality holds:
  \begin{equation}
  \PP\pr{\ens{ \abs{ \sum_{j=1}^nd_j}>x}\cap \ens{\sum_{j=1}^n 
 \pr{ d_j^2+\E{d_j^2\mid \Fca_{j-1}}} 
   \leq y} 
  }\leq 2\exp\pr{-\frac{x^2}{2y}  }.
  \end{equation}   
  
 \end{Proposition}
 
 The following is Lemma~3.2 in \cite{giraudo2019_deviation}.
\begin{Lemma}\label{lem:Lemma_weak_type_estimate}
  Assume that $X$ and $Y$ are two 
  non-negative random variables such that for each positive $x$, 
  we have 
  \begin{equation}\label{eq:weak_type_assumption}
   x\PP\ens{X>x}\leqslant\mathbb 
  E\left[Y \mathbf 1\ens{X\geqslant x}\right].
  \end{equation}
Then for each $t$, the following inequality holds:
  \begin{equation}
   \PP\ens{X>2t}\leqslant \int_1^{+\infty}\PP\ens{Y>st}\mathrm ds.
  \end{equation}
\end{Lemma}

\subsubsection{Facts on Orlicz spaces} 
 
 \begin{Lemma}[Lemma~3.7 in \cite{giraudo2019bounded}]\label{lem:norme_Orlicz_c_fois_fct_YOug}
 Let $p\geq 1$ and $r\geq 0$. Let $\varphi:=\varphi_{p,q}$ and 
 let $a>0$ be a constant. There exists a constant $c$ depending only on 
 $a$, $p$ and $q$ such that for all random variable $X$, 
 \begin{equation}
 \norm{X}_{\varphi}\leq c\norm{X}_{a\varphi}.
 \end{equation}
 \end{Lemma}

\begin{Lemma}[Lemma~3.8 in \cite{giraudo2019bounded}]\label{lem:norm_Orlicz_puissances}
Let $r\geq 0$. There exists a constant $c_r$ such that for any random variable $X$, 
\begin{equation}\label{eq:norm_Orlicz_carre}
\norm{X^2}_{1,r}\leq c_r\norm{X}_{2,r}^2;
\end{equation}
\begin{equation}\label{eq:norm_Orlicz_racine}
\norm{X^{1/2}}_{2,r}\leq c_r\norm{X}_{1,r}^{1/2}.
\end{equation}
\end{Lemma}

\begin{Lemma}[Lemma~3.11 in \cite{giraudo2019bounded}]\label{lem:from_weak_type_to_Orlicz}
For all $p>1$ and $r\geq 0$, there exists a constant $c_{p,r}$ such that if 
 $X$ and $Y$ are two 
  non-negative random variables satisfying for each positive $x$, 
  \begin{equation} 
   x\PP\ens{X>x}\leqslant\mathbb 
  E\left[Y \mathbf 1\ens{X\geqslant x}\right],
  \end{equation}
  then $\norm{X}_{p,r}\leq c_{p,r}\norm{Y}_{p,r}$. 
\end{Lemma} 
 
 \begin{Lemma}\label{lem:control_somme_probas_par_Lpq}
 For any non-negative random variable $X$, $p\geq 1$ and $q\geq 0$, the 
 following inequalities hold, 
 \begin{equation}
 \sum_{k=1}^{+\infty}2^{kp}k^q \PP\ens{X>\frac{2^k}{\sqrt k}}
 \leq c_{p,q}\E{X^p\pr{\ln X}^{q+p/2}\mathbf 1\ens{X> 1}  },
 \end{equation}
  \begin{equation}\label{eq:somme_2puisk_kpuisq}
 \sum_{k=1}^{+\infty}2^{k }k^q \PP\ens{X> 2^{k/2}}
 \leq c_{q}\E{X^2\pr{\ln X}^{q}\mathbf 1\ens{X> 1}  },
 \end{equation}
 where $c_{p,q}$ depends only on $p$ and $q$ and $c_q$ only on $q$.
 \end{Lemma}
 \begin{proof}
 Let $a_k:=2^k/\sqrt k$ and $A_j$ be the event $\ens{a_k<X\leq a_{k+1}}$. Since 
 for all $k\geq 1$, the set  $\ens{X>a_k}$ is the disjoint union of $A_j$, $j\geq k$, we 
 have 
 \begin{equation}
 A:=\sum_{k=1}^{+\infty}2^{kp}k^q \PP\ens{X>\frac{2^k}{\sqrt k}}
 =\sum_{k=1}^{+\infty}2^{kp}k^q \sum_{j\geq k}\PP\pr{A_j}
 =\sum_{j=1}^{+\infty}\PP\pr{A_j} \sum_{k=1}^j2^{kp}k^q.
 \end{equation}
 Since there exists a constant $K_{p,q}$ such that for all $j\geq 1$, 
 $\sum_{k=1}^j2^{kp}k^q\leq K_{p,q}2^{jp}j^q$, it follows that 
 \begin{equation}
 A\leq K_{p,q}\sum_{j=1}^{+\infty}2^{jp}j^q\PP\pr{A_j}.
 \end{equation}
 Writing 
 \begin{equation}
 2^{jp}j^q\PP\pr{A_j}=a_j^pj^{q+p/2}\PP\pr{a_j<X\leq a_{j+1}}
 \leq \E{  X^pj^{q+p/2}\mathbf 1\ens{a_j<X\leq a_{j+1}}   },
 \end{equation}
 the previous estimate becomes 
 \begin{equation}
 A\leq K_{p,q}\sum_{j=1}^{+\infty}
 \E{  X^pj^{q+p/2}\mathbf 1\ens{a_j<X\leq a_{j+1}}   }. 
 \end{equation}
 For $x\in (a_j,a_{j+1}]$, we have in view of $2^j\geq j$ that 
 $2^j\leq x\sqrt j$, hence $2^j\leq x2^{j/2}$ which implies that 
 $j\ln 2\leq 2\ln x$. We end the proof by letting 
 $c_{p,q}:=\pr{2/\ln 2}^{q+p/2}$ 
 and by noticing that $\bigcup_{j\geq 1}A_j\subset \ens{X>1}$.
 
 The proof of \eqref{eq:somme_2puisk_kpuisq} is analogous and consequently 
 omitted.
 \end{proof}
 
 \subsection{Reduction to dyadics}
Let $d$ be a fixed integer and for $0\leq i\leq d-1$ define by  $\N_i$ the elements of 
$\pr{\N\setminus \ens{0}}^d$ whose coordinates $i+1,\dots,d$ are dyadic numbers. More 
formally, 
\begin{equation}
\N_i:=\ens{\gr{n}\in\N^d, \min_{1\leq q\leq d} n_q\geq 1\mbox{ and for all }
i+1\leq j\leq d, \exists k_j\in\N\mbox{ such that }n_j=2^{k_j}   }.
\end{equation}
We also define $\N_d$ as $\N^d$. Notice that $\N_0$ is the set of all the elements of 
$\N^d$ such that all the coordinates are powers of $2$.  The goal of this 
subsection is to show that it suffices to prove 
Theorem~\ref{thm:LIL_orthomartingales} where the supremum over $\N^d$ is replaced by 
the corresponding one over $\N_0$.

\begin{Proposition}\label{prop:reduction_dyadiques_ortho_martingales}
Let $\pr{m\circ T^{\gri}}_{\gri\in\Z^d}$ be  an orthomartingale difference random field with
respect to a commuting filtration 
$\pr{T^{-\gri}\f_{\gr{0}}}_{\gr{i}\in\Z^d}$. Then for all $1<p<2$, the following inequality holds 
\begin{equation}
\norm{
\sup_{\gr{n}\in\N^d}\frac{\abs{S_{\gr{n}}\pr{m}   }}{
\abs{\gr{n}}^{1/2}\prod_{i=1}^dLL\pr{n_i  }^{1/2}}}_p\leq 
c_{p,d}\norm{\sup_{\gr{n}\in\N_0}\frac{\abs{S_{\gr{n}}\pr{m}}}{
\abs{\gr{n}}^{1/2}\prod_{i=1}^dLL\pr{n_i  }^{1/2}}}_p,
\end{equation}
where $c_{p,d}$ depends only on $p$ and $d$.
\end{Proposition}

\begin{Lemma}\label{lem_Mi_Mi-1}
Let $\pr{a_{\gr{n}}}_{\gr{n}\in \N^d}$ be a family of positive numbers such that 
$a_{\gr{n}}\leq a_{\gr{n'}}$ if $\gr{n}\imd \gr{n'}$ and 
\begin{equation}\label{eq:definition_de_c}
c:=\sup_{\gr{n}\in\N^d}\max_{1\leq i\leq d}\frac{a_{\gr{n}+n_i \gr{e_i}   }}{a_{\gr{n}    }}<+\infty. 
\end{equation}
 Assume that $\pr{m\circ T^{\gri}}_{\gr{i}\in\Z^d}$ 
is an orthomartingale difference random field with respect to a commuting filtration 
$\pr{T^{-\gri}\f_{\gr{0}}}_{\gr{i}\in\Z^d}$. Let
\begin{equation}
M_i:=\sup_{\gr{n}\in\N_i}\frac{\abs{S_{\gr{n}}\pr{m}  }}{a_{\gr{n}}}.
\end{equation}
Then for any positive real number number $x$ and any $i\in\ens{0,\dots,d}$, 
\begin{equation}\label{eq:lem_estimate_Mi}
\PP\ens{M_i>x}\leq \int_1^{+\infty}\PP\ens{M_{i-1}> \frac{ux}{2c}     }\mathrm du.
\end{equation}

\end{Lemma}

\begin{proof}
Let $0\leq i\leq d-1$. Define the random variables 
\begin{equation}
Y_N:= \frac1{a_{n_1,\dots,n_{i-1},N,n_{i+1},\dots,n_d}} \sup_{n_1,\dots,n_{i-1}}\sup_{n_{i+1},\dots,n_d}
\abs{S_{n_1,\dots,n_{i-1},N,n_{i+1},\dots,n_d}\pr{m}},
\end{equation}
\begin{equation}
Y'_N:=\frac{a_{n_1,\dots,n_{i-1},N,n_{i+1},\dots,n_d}}{a_{n_1,\dots,n_{i-1},2^{n+1},n_{i+1},\dots,n_d}}  Y_N,\quad  2^n+1\leq N\leq 2^{n+1}.
\end{equation}
and the following events 
\begin{equation}
A_N:=\ens{Y_N>x},  B_0=\emptyset,  B_N:=A_N\setminus \bigcup_{i=0}^{N-1}A_i,
\end{equation}
\begin{equation}\label{eq:definition_de_C_Nn}
C_{N,n}:=\begin{cases}
\bigcup_{i=2^n+1}^NB_i, &\mbox{ if }2^n+1\leq N\leq 2^{n+1};\\
\emptyset, &\mbox{ if } N\leq 2^n  \mbox{ or }  N>  2^{n+1}.
\end{cases}
\end{equation}
In this way, the set $\ens{M_i>x}$ can be expressed as the disjoint union 
$\bigcup_{N\geq 1}B_N$. Hence 
\begin{equation}
\PP\ens{M_i>x}\leq \sum_{N\geq 1}\PP\pr{B_N}=
\sum_{n=0}^{+\infty}\sum_{N=2^n+1}^{2^{n+1}}\PP\pr{B_N}.
\end{equation}
Since $x\mathbf{1}\pr{B_N}\leq Y_N \mathbf{1}\pr{B_N}$, we infer that 
\begin{equation}
x\PP\ens{M_i>x}\leq\sum_{n=0}^{+\infty}\sum_{N=2^n+1}^{2^{n+1}}\E{ 
Y_N\mathbf 1\pr{B_N}
}.
\end{equation}
By definition of $c$ in \eqref{eq:definition_de_c}, we get that 
\begin{equation}\label{eq:estimate_tail_of_Mi}
x\PP\ens{M_i>x}\leq c\sum_{n=0}^{+\infty}\sum_{N=2^n+1}^{2^{n+1}}\E{ 
Y'_N\mathbf 1\pr{B_N}
}.
\end{equation}
Let $n\geq 0$ be fixed. Since $\mathbf 1\pr{B_N}=\mathbf 1\pr{C_{N,n}}-
\mathbf 1\pr{C_{N-1,n}}$ for all $n$ such that $2^n+1\leq N\leq 2^{N+1}$, 
\begin{align*}
\sum_{N=2^n+1}^{2^{n+1}}\E{ 
Y'_N\mathbf 1\pr{B_N}
}&=\sum_{N=2^n+1}^{2^{n+1}}\E{ 
Y'_N\pr{\mathbf 1\pr{C_{N,n}}-
\mathbf 1\pr{C_{N-1,n}}}
}\\
&=\E{\sum_{N=2^n+1}^{2^{n+1}} 
Y'_N \mathbf 1\pr{C_{N,n}}-
\sum_{N=2^n}^{2^{n+1}-1}Y'_{N+1} \mathbf 1\pr{C_{N,n}}}
\\
&\leq \E{Y'_{2^{n+1}} \mathbf 1\pr{C_{2^{n+1},n}} }+
\E{\sum_{N=2^n+1}^{2^{n+1}-1} 
\pr{ Y'_N  -Y'_{N+1} }     \mathbf 1\pr{C_{N,n}}}.
\end{align*}

The set $\mathbf 1\pr{C_{N,n}}$ is measurable with respect to the 
$\sigma$-algebra $\mathcal G_N:=\Fca_{\infty_{[d]\setminus \ens{i} }   +N\gr{e_i}
 }$ and by the orthomartingale difference property of $\pr{m\circ T^{\gri}}_{\gr{i}\in\Z^d}$ the random variable
 $\E{Y'_{N+1}-Y'_N\mid \Gca_N }$ is non-negative and consequently, 
 \begin{align}
 \E{\pr{ Y'_N  -Y'_{N+1} }     \mathbf 1\pr{C_{N,n}}}
 &=\E{\E{\pr{ Y'_N  -Y'_{N+1} }     \mathbf 1\pr{C_{N,n}} \mid \Gca_N}}\\
 &=\E{\mathbf 1\pr{C_{N,n}} \E{Y'_N  -  Y'_{N+1}    \mid \Gca_N }}\leq 0,
 \end{align}
 from which it follows that 
 \begin{equation}
 \sum_{N=2^n+1}^{2^{n+1}}\E{ 
Y'_N\mathbf 1\pr{B_N}
}\leq \E{Y'_{2^{n+1}} \mathbf 1\pr{C_{2^{n+1},n}} }.
 \end{equation}
 The last inequality combined with \eqref{eq:estimate_tail_of_Mi} allows to deduce that 
 \begin{equation}\label{eq:estimate_tail_of_Mi_bis}
x\PP\ens{M_i>x}\leq c\sum_{n=0}^{+\infty}\E{Y'_{2^{n+1}} \mathbf 1\pr{C_{2^{n+1},n}} }.
\end{equation}
Observe that for all $n\geq 0$, the random  variable $Y'_{2^{n+1}}$ is bounded by $M_{i-1}$. Combining this 
with the definition of $C_{N,n}$ given by \eqref{eq:definition_de_C_Nn}, we derive that 
\begin{equation}\label{eq:estimate_tail_of_Mi_ter}
x\PP\ens{M_i>x}\leq c\sum_{n=0}^{+\infty}\E{M_{i-1} \mathbf 1\pr{\bigcup_{k=2^n+1}^{2^{n+1}}B_k  } }.
\end{equation}
Since the family $\ens{B_k,k\geq 1}$ is pairwise disjoint, so is the family 
$\ens{\bigcup_{k=2^n+1}^{2^{n+1}}B_k ,n\geq 0}$. Therefore, using again the fact that 
$\ens{M_i>x}$ can be expressed as the disjoint union 
$\bigcup_{N\geq 1}B_N$, we establish the inequality 
\begin{equation}
x\PP\ens{M_i>x}\leq c\E{M_{i-1}\mathbf 1\ens{M_i>x}   }.
\end{equation}
We estimate the right hand side of the last inequality in the 
following way:
\begin{align*}
 \E{M_{i-1}\mathbf 1\ens{M_i>x}   }&=\int_0^{+\infty} 
 \PP\pr{\ens{M_i>x}\cap \ens{M_{i-1}>t}}\mathrm dt\\
 &\leq \int_0^{x/\pr{2c}} 
 \PP\ens{M_i>x}\ \mathrm dt+\int_{x/\pr{2c}}^{+\infty} 
 \PP \ens{M_{i-1}>t}\mathrm dt\\
 &=\frac{x}{2c}\PP\ens{M_i>x}+\frac{x}{2c}
 \int_{1}^{+\infty} 
 \PP \ens{M_{i-1}>\frac{x}{2c}u}\mathrm du,
\end{align*}
from which \eqref{eq:lem_estimate_Mi} follows.
This ends the proof of Lemma \ref{lem_Mi_Mi-1}.
\end{proof}

\subsection{Proof of Theorem \ref{thm:LIL_orthomartingales}}

We will use the following notations. We define the random variables
\begin{equation}\label{eq:dfn_de_Y_n}
Y_{\gr{n}}:=\frac{\abs{S_{\gr{2^n}}\pr{m}}}{\abs{\gr{2^n}}^{1/2}   
\prod_{q=1}^dL\pr{n_i}^{1/2} },
\end{equation}
\begin{equation}\label{eq:dfn_de_Y_ni}
Y_{\gr{n},i}:=\frac{ S_{\gr{2^{n-\pr{n_i-1}e_i}}}\pr{m}  }{\abs{\gr{2^{n-n_ie_i}}}^{1/2}   
\prod_{q=1,q\neq i}^dL\pr{n_q}^{1/2} },
\end{equation}
\begin{equation}\label{eq:dfn_de_Z}
Z_i:=\sup_{\gr{n}\in\N^d}\abs{Y_{\gr{n},i}}.
\end{equation}

\begin{Lemma}\label{lem:lien_dim_d_dim_d-1}
Let $\pr{m\circ T^{\gri},T^{-\gri}\Fca_{\gr{0}}}_{\gr{i}\in\Z^d}$ be a strictly stationary
orthomartingale difference random field. For all integer $N$ and all $x\geq e^{2d+2}$, the 
following inequality holds:
\begin{multline}\label{eq:lien_dim_d_dim_d-1}
\PP\ens{\sup_{\gr{n}\in\N^d} Y_{\gr{n}}>x}\leq 
2^{Nd}\PP\ens{\abs{m}>x2^{-Nd}    }+d N^{d/2}x^{-\ln N}\\
+8\sum_{i=1}^d\int_{1/\sqrt 2}^{+\infty}v\PP\ens{Z_i > \frac{x}{\sqrt{\ln x}}v/2
}\mathrm dv.
\end{multline}
\end{Lemma}

\begin{proof}
Define the events 
\begin{equation}
A_{\gr{n}}:=\ens{Y_{\gr{n}}> x},
\end{equation}
\begin{equation}
B_{\gr{n},i}:=\ens{ \frac1{2^{n_i}}\sum_{j=1}^{2^{n_i}}U^{j\gr{e_i}}
\pr{Y_{\gr{n,i}}^2
+\E{Y_{\gr{n,i}}^2 \mid \Fca_{\infty\gr{1_{[d]\setminus \ens{i}   }}- \gr{e_i}  } } }
 \leq \frac{x^2}{\ln x}   }, B_{\gr{n}}:=\bigcap_{i=1}^dB_{\gr{n},i}.
\end{equation}
 Denoting by $J_N$ the set of elements of $\N^d$ such that 
at least one coordinate is bigger than $N+1$, we have  
\begin{equation}\label{eq:estimation_de_la_prb_union_An}
\PP\ens{\bigcup_{\gr{n}\in \N^d}A_{\gr{n}}}
\leq \PP\pr{\bigcup_{\gr{1}\imd \gr{n}\imd N\gr{1}}A_{\gr{n}} }
+\sum_{\gr{n}\in J_N}\PP\pr{ A_{\gr{n}}\cap B_{\gr{n}}}+
\sum_{i=1}^d\PP\pr{\bigcup_{\gr{n}\in \N^d} B_{\gr{n},i} ^c    }.
\end{equation} 
Observe that for $\gr{1}\imd \gr{n}\imd N\gr{1}$, the
following inclusions hold
\begin{align}
A_{\gr{n}}&\subset \ens{ \sum_{\gr{1}\imd\gr{i}\imd\gr{2^n}  } \abs{ 
m\circ T^{\gri}
 }>x \abs{\gr{2^n}}^{1/2}   }\\
 &\subset \ens{ \sum_{\gr{1}\imd\gr{i}\imd 2^N\gr{1}  } \abs{ 
m\circ T^{\gri} 
 }>x    }\\
 &\subset \bigcup_{\gr{1}\imd\gr{i}\imd 2^N\gr{1}}
 \ens{  \abs{ 
m\circ T^{\gri} 
 }>x2^{-Nd}    }.
\end{align}
  Hence 
 \begin{equation}\label{eq:estimation_somme_PAn}
 \PP\pr{\bigcup_{\gr{1}\imd \gr{n}\imd N\gr{1}}A_{\gr{n}} }
 \leq 2^{Nd}\PP\ens{\abs{m}>x2^{-Nd}    }.
\end{equation}  
Let us control $\PP\pr{ A_{\gr{n}}\cap B_{\gr{n}}}$. First observe that 
\begin{equation}\label{eq:estimate_PAn_inter_B_n}
\PP\pr{ A_{\gr{n}}\cap B_{\gr{n}}}\leq 
\min_{1\leq i\leq d}\PP\pr{ A_{\gr{n}}\cap B_{\gr{n},i}}.
\end{equation}
Then, in order to control $\PP\pr{ A_{\gr{n}}\cap B_{\gr{n},i}}$ for a fixed 
$i\in [d]$, we apply Proposition~\ref{prop:inegalite_deviation_martingales} in 
following setting:
\begin{equation}
d_j:=U^{j\gr{e_i}}Y_{\gr{n},i}, \widetilde{\Fca_j}:=
\Fca_{\infty\gr{1_{[d]\setminus \ens{i}   }}+\pr{j-1} \gr{e_i}  }, 
\widetilde{x}:= x2^{n_i/2}L\pr{n_i}^{1/2},
\widetilde{y}:=2^{n_i}x^2/\ln x .
\end{equation}
This leads to the following estimate  
\begin{equation}
\PP\pr{ A_{\gr{n}}\cap B_{\gr{n},i}}\leq 2\exp\pr{-\frac 12L\pr{n_i}\ln x}
\end{equation}
and plugging this into \eqref{eq:estimate_PAn_inter_B_n} gives 
\begin{equation}\label{eq:estimate_PAn_inter_B_n_2}
\PP\pr{ A_{\gr{n}}\cap B_{\gr{n}}}\leq 2
\pr{\max_{1\leq i\leq d} n_i}^{-\frac{\ln x}{2}}.
\end{equation}
For a fixed positive integer $\ell$, the number of elements of $\N^d$ such that 
$\max_{1\leq i\leq d} n_i=\ell$ does not exceed $d\ell^{d-1}$. Hence 
\begin{equation}
\sum_{\gr{n}\in J_N}\PP\pr{ A_{\gr{n}}\cap B_{\gr{n}}}
\leq \sum_{\ell\geq N}dm^{d-1}\ell^{-\frac{\ln x}{2}}
\end{equation}
and the latter sum can be estimated by 
\begin{equation}
d\frac{1}{\frac{\ln x}{2}-d}N^{d-\frac{\ln x}{2}}=d\frac{1}{\frac{\ln x}{2}-d}
N^{d/2}x^{-\ln N}.
\end{equation}
Hence 
\begin{equation}
\sum_{\gr{n}\in J_N}\PP\pr{ A_{\gr{n}}\cap B_{\gr{n}}}
\leq d\frac{1}{\frac{\ln x}{2}-d}
N^{d/2}x^{-\ln N}.
\end{equation}
Since $x\geq e^{2d+2}$, we deduce the estimate 
\begin{equation}\label{eq:estimate_somme_PAn_inter_B_n}
\sum_{\gr{n}\in J_N}\PP\pr{ A_{\gr{n}}\cap B_{\gr{n}}}
\leq d N^{d/2}x^{-\ln N}.
\end{equation}

In order to bound the third term of the right hand side in 
\eqref{eq:estimation_de_la_prb_union_An}, we need the following inequality, 
valid for any map $Q\colon \mathbb L^{\infty}\to \mathbb L^{\infty}$ such that 
$\norm{Qf}_1\leq \norm{f}_1$ and $\norm{Qf}_\infty\leq \norm{f}_\infty$ for all $f\in \mathbb L^\infty$:
\begin{equation}
\PP\ens{\sup_{N\geq 1}\frac 1N\abs{\sum_{j=1}^NQ^jf  }>y}\leq \int_{1/2}^{+\infty}\PP\ens{ \abs f>yu}\mathrm{d}u
\end{equation}
(see Lemma~6.1 in \cite{MR797411}).
We 
apply it for any 
$i\in [d]$ to the following setting:
$f=Y_i^2$, $Q\colon g\mapsto \pr{U^{\gr{e_i}}g+U^{\gr{e_i}}\E{g\mid\Fca_{\infty\gr{1_{[d]\setminus \ens{i}   }}- \gr{e_i}  }     }}/2$,
and $y:=\frac{x^2}{\ln x}$.
 
This leads to 
\begin{equation}
\PP\pr{\bigcup_{\gr{n}\in \N^d} B_{\gr{n},i} ^c    }
\leq 4\int_{1/2}^{+\infty}\PP\ens{Z_i^2> \frac{x^2}{\ln x}u/4}\mathrm du,
\end{equation}
and after the substitution $v=\sqrt u$, the latter term becomes 
\begin{equation}
8\int_{1/\sqrt 2}^{+\infty}v\PP\ens{Z_i > \frac{x}{\sqrt{\ln x}}v/2}\mathrm dv.
\end{equation}
Hence 
\begin{equation}\label{eq:estimation_prob_Bni}
\sum_{i=1}^d\PP\pr{\bigcup_{\gr{n}\in \N^d} B_{\gr{n},i} ^c    }
\leq 8\sum_{i=1}^d\int_{1/\sqrt 2}^{+\infty}v\PP\ens{Z_i > \frac{x}{\sqrt{\ln x}}v/2
}\mathrm dv.
\end{equation}
We end the proof of Lemma~\ref{lem:lien_dim_d_dim_d-1} by combining 
\eqref{eq:estimation_de_la_prb_union_An}, \eqref{eq:estimation_somme_PAn}, 
\eqref{eq:estimate_somme_PAn_inter_B_n} and 
\eqref{eq:estimation_prob_Bni}. 
\end{proof}

\begin{Lemma}\label{lem:recurrence_moments_p}
Let  
$Z:=\sup_{\gr{n}\in \N_0}Y_{\gr{n}}$, where 
$Y_{\gr{n}}$ is defined by \eqref{eq:dfn_de_Y_n}
 and $Z_i$ given by \eqref{eq:dfn_de_Z}.
 Let $1\leq p<2$. There 
exists a constant $C_{p,d}$ depending only on $p$ and $d$ such that 
for all strictly stationary
orthomartingale difference random field $\pr{m\circ T^{\gri},T^{-\gri}\Fca_{\gr{0}}}_{\gr{i}\in\Z^d}$,
the following inequality holds:
\begin{equation}\label{eq:rec_control_norm_pr}
\norm{Z}_{p}\leq C_{p,d}\max_{1\leq i\leq d}\norm{Z_i}_{2}.
\end{equation}
\end{Lemma}

\begin{proof}
The proof follows the same lines as that of Lemma~\ref{lem:lien_dim_d_dim_d-1}. 
Let $x$ be such that $x^p/2-d\geq pd/\pr{2-p}$. Define the events 
\begin{equation}
A_{\gr{n}}:=\ens{Y_{\gr{n}}> x},
\end{equation}
\begin{equation}
B_{\gr{n},i}:=\ens{ \frac1{2^{n_i}}\sum_{j=1}^{2^{n_i}}U^{j\gr{e_i}}
\pr{Y_{\gr{n,i}}^2
+\E{Y_{\gr{n,i}}^2 \mid \Fca_{\infty\gr{1_{[d]\setminus \ens{i}   }}- \gr{e_i}  } } }
 \leq  x^{p}   }, B_{\gr{n}}:=\bigcap_{i=1}^dB_{\gr{n},i}.
\end{equation}
 Denoting by $J_N$ the set of elements of $\N^d$ such that 
at least one coordinate is bigger than $N+1$, we have  
\begin{equation}\label{eq:estimation_de_la_prb_union_An_bis}
\PP\ens{\bigcup_{\gr{n}\in \N^d}A_{\gr{n}}}
\leq \sum_{\gr{1}\imd \gr{n}\imd N\gr{1}}\PP\pr{A_{\gr{n}} }
+\sum_{\gr{n}\in J_N}\min_{1\leq i\leq d}\PP\pr{ A_{\gr{n}}\cap B_{\gr{n},i}}+
\sum_{i=1}^d\PP\pr{\bigcup_{\gr{n}\in \N^d} B_{\gr{n},i} ^c    }.
\end{equation} 
For $\gr{1}\imd \gr{n}\imd N\gr{1}$, we control $\PP\pr{A_{\gr{n}} }$  
by using Chebyshev's and Doob's inequality in order to get 
\begin{equation}
\sum_{\gr{1}\imd \gr{n}\imd N\gr{1}}\PP\pr{A_{\gr{n}} }\leq 2^{d}
N^d x^{-2}.
\end{equation}
By similar arguments as in the proof of Lemma~\ref{lem:lien_dim_d_dim_d-1}, 
we obtain that 
\begin{equation}
\sum_{\gr{n}\in J_N}\min_{1\leq i\leq d}\PP\pr{ A_{\gr{n}}\cap B_{\gr{n},i}}
\leq \sum_{\ell\geq N+1}d\ell^{d-1}\ell^{-x^{2-q}/2 }\leq d\frac{1}{x^{p}/2-d}
N^{d-x^{p}/2}.
\end{equation}
Using inequality $x^p/2-d\geq pd/\pr{2-p}$ yields 
\begin{equation}
\sum_{\gr{n}\in J_N}\min_{1\leq i\leq d}\PP\pr{ A_{\gr{n}}\cap B_{\gr{n},i}}
\leq d\frac{2-p}pN^{-\frac{p}{2-p}}.
\end{equation}
Moreover, an application of the maximal ergodic theorem gives that 
\begin{equation}
\sum_{i=1}^d\PP\pr{\bigcup_{\gr{n}\in \N^d} B_{\gr{n},i} ^c    }\leq 
x^{-p}\sum_{i=1}^d\E{Z_i^2}.
\end{equation}
The combination of the previous three estimates gives that for all positive $x$
\begin{equation}
\PP\ens{Z>x}\leq c_{p,d}\min\ens{1, N^d x^{-2}+ 
N^{pd/\pr{2-p}}+ x^{2-q}\sum_{i=1}^d\E{Z_i^2} }.
\end{equation} 
We choose $N:=\left[x^{\frac{2-p}d}\right]$; in this way  
\begin{equation}
\PP\ens{Z>x}\leq   c_{p,d}\min\ens{1,  x^{-p}+ 
 x^{-p}\sum_{i=1}^d\E{Z_i^2} }.
\end{equation} 
If $\sum_{i=1}^d\E{Z_i^2} \leq 1$, this gives that $\norm{Z}_{p,w}
\leq  c_{p,d}$ and we replace $m$ by $m/ 
\sqrt{\sum_{i=1}^d\E{Z_i^2} }$ to get the general case. 

This ends the proof of Lemma~\ref{lem:recurrence_moments_p}.
\end{proof}

\begin{Lemma}\label{lem:recurrence_moments_2}
Let  
$Z:=\sup_{\gr{n}\in \N_0}Y_{\gr{n}}$, where 
$Y_{\gr{n}}$ is defined by \eqref{eq:dfn_de_Y_n}
 and $Z_i$ given by \eqref{eq:dfn_de_Z}. There exists a constant $C_{r,d}$ depending only on $r$ and $d$ such that 
for all strictly stationary
orthomartingale difference random field $\pr{m\circ T^{\gri},T^{-\gri}\Fca_{\gr{0}}}_{\gr{i}\in\Z^d}$,
the following inequality holds:
\begin{equation}\label{eq:rec_control_norm_2r_par_norm2r+2}
\norm{Z}_{2,r}\leq C_{r,d}\max_{1\leq i\leq d}\norm{Z_i}_{2,r+2}.
\end{equation}
\end{Lemma}

\begin{proof}
In this proof, $c$ will denote a constant depending only on $r$ and $d$ and 
which may change from line to line.
We start from the equality 
\begin{equation}
\E{\varphi_{2,r}\pr{Z}   }=\int_0^{+\infty}\varphi'_{2,r}\pr{x}\PP\ens{Z>x}\mathrm dx.
\end{equation}
Since $0\leq \varphi'_{2,r}\pr{x}\leq c x\pr{
\log\pr{1+x}}^r$ and $\varphi'_{2,r}\pr{x}\leq 
c $ for $x\leq e^{2\pr{d+1}}$, we derive that 
\begin{equation}
\E{\varphi_{2,r}\pr{Z}   }\leq c +c 
\int_{e^{2\pr{d+1}}}^{+\infty}x\pr{
\log\pr{1+x}}^r\PP\ens{Z>x}\mathrm dx.
\end{equation}
An application of Lemma~\ref{lem:lien_dim_d_dim_d-1} with $N$ such that 
$\ln N>2$ (for example $N=10$) gives that for some constant $c $, 
\begin{multline}\label{eq:recurrence_moments_etape_intermediaire}
\E{\varphi_{2,r}\pr{Z}   }\leq c \left(1
+\int_{e^{2\pr{d+1}}}^{+\infty}x \pr{
\log\pr{1+x}}^r\PP\ens{\abs{m}>x}\mathrm dx 
  \right.\\\left. +\int_{e^{2\pr{d+1}}}^{+\infty}x^{1-\ln 10}\pr{
\log\pr{1+x}}^r\mathrm dx\right.\\
 \left.+\sum_{i=1}^d\int_{e^{2\pr{d+1}}}^{+\infty}x \pr{
\log\pr{1+x}}^r\int_{1/\sqrt 2}^{+\infty}v\PP\ens{
Z_i>\frac{x}{\sqrt{\ln x}}\frac v2
}\mathrm dv     \mathrm dx
\right).
\end{multline}
The second term does not exceed $\E{\varphi_{2,r}\pr{\abs{m}}}$.
The third term of \eqref{eq:recurrence_moments_etape_intermediaire} 
is a constant depending on $p$ and $r$. Therefore, 
\begin{multline}\label{eq:recurrence_moments_etape_intermediaire_2}
\E{\varphi_{2,r}\pr{Z}   }\leq c \left(1+\E{\varphi_{2,r}\pr{\abs{m}}}\right.\\
\left.+
\sum_{i=1}^d\int_{1/\sqrt 2}^{+\infty}\int_{e^{2\pr{d+1}}}^{+\infty}x \pr{
\log\pr{1+x}}^rv\PP\ens{
Z_i>\frac{x}{\sqrt{\ln x}}\frac v2
}    \mathrm dx\mathrm dv 
\right).
\end{multline}
Bounding the integral over $x$ by the corresponding one on $\pr{2,+\infty}$, cutting 
this interval into intervals of the form $\left(2^k,2^{k+1}\right]$, we end up with the 
inequality 
\begin{equation}\label{eq:recurrence_moments_etape_intermediaire_3}
\E{\varphi_{2,r}\pr{Z}   }\leq c \left(1+\E{\varphi_{2,r}\pr{\abs{m}}}+
\sum_{i=1}^d\int_{1/\sqrt 2}^{+\infty}\sum_{k=1}^{+\infty}2^{k } 
k^rv\PP\ens{
Z_i>\frac{2^k}{\sqrt{k}}\frac v2
}    \mathrm dv 
\right).
\end{equation}
We apply Lemma~\ref{lem:control_somme_probas_par_Lpq}
for a fixed $i\in\ens{1,\dots,d}$ and $v>1/\sqrt 2$ to $X:= 2Z_i/v$ in order to 
obtain 
\begin{multline}\label{eq:recurrence_moments_etape_intermediaire_4}
\E{\varphi_{2,r}\pr{Z}   }\leq c \Big(1+\E{\varphi_{2,r}\pr{\abs{m}}} \\
 +
\sum_{i=1}^d\int_{1/\sqrt 2}^{+\infty} v\E{\pr{2Z_i/v}^2 
\pr{\ln\pr{1+2Z_i/v   }}^{r+1}\mathbf{1}\ens{2Z_i/v>1   }
 }   \mathrm dv 
\Big).
\end{multline}

Switching the integral and the expectation,  we first have to bound the random variable 
\begin{equation}
Z'_i:=\int_{1/\sqrt 2}^{+\infty}v\pr{2Z_i/v}^2 
\pr{\ln\pr{1+2Z_i/v   }}^{r+1}\mathbf{1}\ens{2Z_i/v>1   }\mathrm dv.
\end{equation}
A first observation is that if $2Z_i<1/\sqrt 2$, the random variable $Z'_i$ vanishes, 
hence 
\begin{equation}
Z'_i\leq \mathbf{1}\ens{2Z_i>1/\sqrt 2   }\int_{1/\sqrt 2}^{2Z_i}v\pr{2Z_i/v}^2 
\pr{\ln\pr{1+2Z_i/v   }}^{r+1}\mathrm dv.
\end{equation}

\begin{equation}
Z'_i\leq \mathbf{1}\ens{2Z_i>1/\sqrt 2   }\pr{2Z_i}^2\int_{1/\sqrt 2}^{2Z_i}v^{-1}  
\pr{\ln\pr{1+2Z_i/v   }}^{r+1}\mathrm dv
\end{equation}
and the integral is bound by a constant depending on $r$ times 
$\pr{\ln\pr{1+Z_i}}^{r+2}$, hence 
\begin{equation}
Z'_i\leq  c Z_i^2\pr{\ln\pr{1+Z_i}}^{r+2}.
\end{equation}

We obtain 
\begin{equation}\label{eq:control_moment_phi2r_de_Z}
\E{\varphi_{2,r}\pr{Z}   }\leq c \pr{1+\E{\varphi_{2,r}\pr{\abs{m}}}
+\max_{1\leq i\leq d} \E{\varphi_{2,r+2}\pr{Z_i} }
}.
\end{equation}

 Consider $\lambda\geq \max_{1\leq i\leq d}\norm{Z_i}_{2,r+2}$ and also greater 
 than $\norm{m}_{2,r}$. Replacing $m$ by $m/\lambda$ 
in \eqref{eq:control_moment_phi2r_de_Z}  yields 
 \begin{equation}
 \E{\varphi_{2,r}\pr{Z/\lambda}   }\leq 3c.
 \end{equation}
Letting $\varphi:=\varphi_{2,r}/\pr{3c}$ gives that 
\begin{equation}
\norm{Z}_{\varphi}\leq \max_{1\leq i\leq d}\norm{Z_i}_{2,r+2}+\norm{m}_{2,r}.
\end{equation}
Then an application of Lemma~\ref{lem:norme_Orlicz_c_fois_fct_YOug} gives
\begin{equation}
\norm{Z}_{2,r}\leq c\pr{\max_{1\leq i\leq d}\norm{Z_i}_{2,r+2}+\norm{m}_{2,r}}.
\end{equation}
Finally, noticing that $\norm{m}_{2,r}\leq \norm{Z_1}_{2,r+2}$ gives 
\eqref{eq:control_moment_phi2r_de_Z}.   This ends 
the proof of Lemma~\ref{lem:recurrence_moments_2}.
\end{proof}
 
\begin{proof}[End of the proof of Theorem~\ref{thm:LIL_orthomartingales}]
We start by proving \eqref{eq:norme_2_fct_max_orthomartingale} by induction 
on the dimension. For $d=1$, this follows from Lemma~\ref{lem:recurrence_moments_2}. 

Assume that \eqref{eq:norme_2_fct_max_orthomartingale} holds
 for all stationary orthomartingale difference $\pr{d-1}$-dimensional random 
fields (with $d\geq 2$) and all $r\geq 0$. Using Lemma~\ref{lem:recurrence_moments_2}, 
we get that for all $d$ dimensional strictly stationary orthomartingale difference 
random fields,   
$\norm{M}_{2,r}\leq c_{r,d}\max_{1\leq i\leq d}\norm{Z_i}_{2,r+2}$.
By the induction hypothesis applied with $\widetilde{r}:=r+2$, we get that 
 $\norm{Z_i}_{2,r+2}\leq 
\norm{m}_{2,r+2+ 2\pr{d-1}}$, which
 gives \eqref{eq:norme_2_fct_max_orthomartingale}. 

Let us show \eqref{eq:norme_p_fct_max_orthomartingale}.
By Lemma~\ref{lem:recurrence_moments_p}, we derive that 
$\norm{M}_p\leq c_{p,d}\max_{1\leq i\leq d}\norm{Z_i}_{2}$.
Using \eqref{eq:norme_2_fct_max_orthomartingale},
 we derive that  
$\norm{Z_i}_{2}\leq \norm{m}_{2, 2\pr{d-1}}$, from which 
\eqref{eq:norme_p_fct_max_orthomartingale} follows.

This ends the proof of Theorem~\ref{thm:LIL_orthomartingales}.
\end{proof}

\subsection{Proof of Theorem~\ref{thm:Hannan_con}}

The conditions of theorem imply that $f=\lim_{N\to +\infty}
\sum_{-N\gr{1}\imd \gri\imd N\gr{1}}\pi_{\gri}\pr{f}$ almost surely.
Therefore, for each $\grn\in\N^d$, inequalities
\begin{equation}
\abs{S_{\grn}\pr{f}}
\leq \sum_{-N\gr{1}\imd \gri\imd N\gr{1}}\abs{S_{\grn}\pr{   \pi_{\gri}\pr{f}}}
\leq \sum_{ \gri\in\Z^d}\abs{S_{\grn}\pr{   \pi_{\gri}\pr{f}}}
\end{equation}
hold almost surely. Hence 
\begin{equation}
\norm{M\pr{f}}_p\leq \sum_{\gri\in\Z^d}
\norm{  M\pr{   \pi_{\gri }\pr{f}}}_p.
\end{equation}
Since $\pr{U^{\grj} \pi_{\gri }\pr{f}}_{\grj\in\Z^d}$ is an orthomartingale 
difference random field with respect to the completely commuting filtration 
$\pr{T^{-\grj-\gri}\f_{\gr{0}}    }_{\gri\in\Z^d}$, an application of 
Theorem~\ref{thm:LIL_orthomartingales} ends the proof of 
Theorem~\ref{thm:Hannan_con}.

\subsection{Proof of the results of Section~\ref{sec:MW}}

\begin{proof}[Proof of Proposition~\ref{prop:inegalite_presque_sure_dim_d}]
The proof with be done by induction on the dimension $d$. 

In dimension $1$, the right hand side of \eqref{eq:almost_sure_inequality_dim_d} reads 
\begin{equation}\label{eq:membre_de_droite_dim1}
\sum_{k=0}^n\max_{1\leq i\leq 2^{n-k}}
\abs{\sum_{\ell=0}^{i-1} d_{k,\ens{1}}\circ T^{2^{k_1}\ell}     }
+\sum_{k=0}^n\max_{0\leq \ell\leq 2^{n-k}}
\abs{  d_{k,\emptyset}\circ T^{2^{k_1}\ell}     },
\end{equation}
where 
\begin{equation}
d_{k,\emptyset}= \proj{S_{2^k}\pr{f}}{-2^k  },
\end{equation}
\begin{equation}
d_{0,\ens{1}}=f-\proj{f}{-1},
\end{equation}
\begin{equation}
d_{k,\ens{1}}=\proj{S_{2^k}\pr{f}}{-2^{k-1}}-
\proj{S_{2^k}\pr{f}}{-2^{k}},
\end{equation}
and the term in \eqref{eq:membre_de_droite_dim1} is greater than 
the right hand side of \eqref{eq:almost_sure_inequality_dim_1}. 

Let $d\geq 2$ and suppose 
that Proposition~\ref{prop:inegalite_presque_sure_dim_d} holds 
for all $d'$-dimensional random fields where $1\leq d'\leq d-1$. 
Let $T$ be a measure preserving $\Z^d$-action on a probability space 
$\pr{\Omega,\Fca,\mu}$. Let $\Fca_{\gr{0}}\subset\Fca$ 
be a sub-$\sigma$-algebra such that $T^{\gr{e_q}}\Fca_{\gr{0}}\subset 
\Fca_{\gr{0}}$ for all $q\in [d]$ and the filtration 
$\pr{T^{-\gri}\Fca_{\gr{0}}}_{\gri\in\Z^d}$ is commuting. Finally let $f$ be 
an $\Fca_{\gr{0}}$-measurable function and $\grn\in\N^d$. Let $\grj$ be 
such that $\gr{1}\imd\gr{j}\imd \gr{2^n}$. Observe that 
\begin{equation}\label{eq:ineg_1_demo_ineg_max}
\abs{S_{\grj}\pr{f}}\leq \max_{ 
\substack{1\leq i_q\leq 2^{n_q}\\ 
1\leq q\leq d-1
}
}\abs{S_{i_1,\dots,i_{d-1},j_d}\pr{f}}.
\end{equation} 
We apply the $d-1$ dimensional case in the following setting:
\begin{itemize}
\item $\widetilde{\grn}:=\sum_{q=1}^{d-1}n_q\gr{e_q}$;
\item $\widetilde{T}^{\gri}=T^{\sum_{q=1}^{d-1}i_q\gr{e_q}}$, $\gri\in \Z^{d-1}$;
\item $\widetilde{\Fca_{\gr{0}}}:= \bigvee_{i_d\in\Z} T^{i_d\gr{e_d}}\Fca_{\gr{0}}$;
\item $\widetilde{f}:= \sum_{\ell_d=0}^{j_d-1}f\circ T^{\ell_d\gr{e_d}}$.
\end{itemize}
In view of \eqref{eq:ineg_1_demo_ineg_max}, we obtain that 
\begin{equation}\label{eq:ineg_2_demo_ineg_max}
\abs{S_{\grj}\pr{f}}\leq \sum_{\gr{0}\imd\gr{k}\imd\widetilde{\grn}}
\sum_{I\subset [d-1]} \max_{\gr{1}-\gr{1}_{I}\imd\gri\imd 
\gr{2^{\widetilde{\grn}-k}}}\abs{S_{\gri}^I\pr{\widetilde{T}^{\gr{2^k}   },
\widetilde{d_{\grk,I}   }  }},
\end{equation}
where 
\begin{equation}
\widetilde{d_{\grk,I}   } := 
\sum_{I''\subset    I\setminus Z\pr{\grk}}
\sum_{I'\subset  I\cap Z\pr{\grk} }
\pr{-1}^{\abs{I'}+\abs{I''}}
\proj{S_{2^{\grk+\gr{1_{Z\pr{k}}}-\gr{1_I}}}\pr{ \widetilde{f}}  }{-2^{\grk -\gr{1_{I''}}}-\gr{1_{I'}}+\infty \gr{1_{\ens{d}}   }}.
\end{equation}
Observe that for all $\ell_1,\dots,\ell_{d-1}\in\Z$, the $\sigma$-algebra 
$\Fca_{\ell_1,\dots,\ell_{d-1},\infty}$ is invariant by $T^{\gr{e_d}}$. 
Consequently, we can write 
\begin{equation}
\widetilde{d_{\grk,I}   } =\sum_{\ell_d=0}^{j_d-1}f_{\grk,I}\circ 
T^{\ell_d \gr{e_d}},
\end{equation}
where 
\begin{equation}
f_{\grk,I}=\sum_{I''\subset    I\setminus Z\pr{\grk}}
\sum_{I'\subset  I\cap Z\pr{\grk} }
\pr{-1}^{\abs{I'}+\abs{I''}}
\proj{S_{2^{\grk+\gr{1_{Z\pr{k}}}-\gr{1_I}}}\pr{ f}  }{-2^{\grk -\gr{1_{I''}}}-\gr{1_{I'}}+\infty \gr{1_{\ens{d}}   }}.
\end{equation}
Using commutativity of the filtration $\pr{T^{-\gri}\Fca_{\gr{0}}}_{\gri\in\Z^d}$ and 
$\Fca_{\gr{0}}$-measurabiility of $f$, we derive that 
\begin{equation}
\proj{S_{2^{\grk+\gr{1_{Z\pr{k}}}-\gr{1_I}}}\pr{ f}  }{-2^{\grk -\gr{1_{I''}}}-
\gr{1_{I'}}+\infty \gr{1_{\ens{d}}   }}=
\proj{S_{2^{\grk+\gr{1_{Z\pr{k}}}-\gr{1_I}}}\pr{ f}  }{-2^{\grk -\gr{1_{I''}}}-\gr{1_{I'}}
  }.
\end{equation}
Hence 
\begin{equation}
f_{\grk,I}=\sum_{I''\subset    I\setminus Z\pr{\grk}}
\sum_{I'\subset  I\cap Z\pr{\grk} }
\pr{-1}^{\abs{I'}+\abs{I''}}
\proj{S_{2^{\grk+\gr{1_{Z\pr{k}}}-\gr{1_I}}}\pr{ f}  }{-2^{\grk -\gr{1_{I''}}}-\gr{1_{I'}}}.
\end{equation}
Since for all $I\subset [d-1]$ and $\gri\in \Z^{d-1}$, the operators 
$S_{\gri}^I$ and $T^{\gr{e_d}}$ commute, the previous rewriting of 
$\widetilde{d_{\grk,I}   }$ combined with \eqref{eq:ineg_2_demo_ineg_max} yields 
\begin{equation}
\label{eq:ineg_3_demo_ineg_max}
\abs{S_{\grj}\pr{f}}\leq \sum_{\gr{0}\imd\gr{k}\imd\widetilde{\grn}}
\sum_{I\subset [d-1]} \max_{\gr{1}-\gr{1}_{I}\imd\gri\imd 
\gr{2^{\widetilde{\grn}-k}}}\abs{
\sum_{\ell_d=0}^{j_d-1}\pr{
S_{\gri}^I\pr{\widetilde{T}^{\gr{2^k}   },
f_{\grk,I}     } } \circ 
T^{\ell_d \gr{e_d}}   }.
\end{equation}
Fix $\grk\in\Z^d$ such that $\gr{0}\imd\gr{k}\imd\widetilde{\grn}$, 
$\sum_{I\subset [d-1]}$, $\gri\in\Z^d$ such that 
$\gr{1}-\gr{1}_{I}\imd\gri\imd 
\gr{2^{\widetilde{\grn}-k}}$. We apply the result of
 Proposition~\ref{prop:inegalite_presque_sure_dim_d} to the one dimensional case 
 in the following setting:
 \begin{itemize}
 \item $\widetilde{\widetilde{n}}=n_d$,
 \item $\widetilde{\widetilde{T}}^{\ell_d}=T^{\ell_d\gr{e_d}}$, 
 \item $\widetilde{\widetilde{\Fca_{\gr{0}   }}}:= \Fca_{\infty\gr{1_{[d-1]}}   }$,
 \item $\widetilde{\widetilde{f}}:=S_{\gri}^I\pr{\widetilde{T}^{\gr{2^k}   },
f_{\grk,I}     } $.
 \end{itemize}
We get 
\begin{multline}
\abs{
\sum_{\ell_d=0}^{j_d-1}\pr{
S_{\gri}^I\pr{\widetilde{T}^{\gr{2^k}   },
f_{\grk,I}     } } \circ 
T^{\ell_d \gr{e_d}}   }\leq \sum_{k_d=0}^{n_d}
\max_{1\leq i_d\leq 2^{n_d-k_d}}
\abs{\sum_{\ell_d=0}^{i_d-1} \widetilde{\widetilde{d_{k_d,\ens{d}}}    }    }\\
+\sum_{k_d=0}^{n_d}
\max_{0\leq  \ell_d\leq 2^{n_d-k_d}}
\abs{  \widetilde{\widetilde{d_{k_d,\emptyset}}    }\circ T^{2^{k_d}\ell_d}    },
\end{multline}
where 
\begin{equation}
\widetilde{\widetilde{d_{0,\ens{d}}}    }=
\proj{ \widetilde{\widetilde{f}}      }{ \infty\gr{1_{[d-1]}}     }
-\proj{ \widetilde{\widetilde{f}}      }{ \infty\gr{1_{[d-1]}} -\gr{e_d}    },
\end{equation}
\begin{equation}
\widetilde{\widetilde{d_{k_d,\ens{d}}}    }=
\proj{ S_{2^{k_d} \gr{e_d}   }   \pr{ \widetilde{\widetilde{f}}}      }{ \infty\gr{1_{[d-1]}} -2^{k_d-1}
\gr{e_d}    }
-\proj{ S_{2^{k_d}\gr{e_d}    }\pr{ \widetilde{\widetilde{f}}}    }{ \infty\gr{1_{[d-1]}} -2^{k_d}\gr{e_d}    },
\end{equation}
\begin{equation}
\widetilde{\widetilde{d_{k_d,\emptyset}}    }=
\proj{ S_{2^{k_d} \gr{e_d}   }\pr{ \widetilde{\widetilde{f}}}       }{ \infty\gr{1_{[d-1]}} -2^{k_d}\gr{e_d}    }.
\end{equation}
Simplifying the expression gives the wanted result.
\end{proof}

\begin{proof}[Proof of Theorem~\ref{thm:MW_condition}]
Starting from Proposition~\ref{prop:inegalite_presque_sure_dim_d}, we derive that 
\begin{equation}\label{eq:demo_MW_controle_de_Mf}
M\pr{f}\leq c_d\sup_{\grn\smd \gr{0}} \frac 1{\abs{\gr{2^n}}^{1/2}\prod_{q=1}^d \pr{L\pr{n_q}}^{1/2}}
\sum_{\gr{0}\imd\gr{k}\imd\gr{n}}
\sum_{I\subset [d]}
\max_{\gr{1}-\gr{1}_{I}\imd\gri\imd 
\gr{2^{n-k}}}\abs{S_{\gri}^I\pr{T^{\gr{2^k}   },d_{\grk,I}   }  },
\end{equation}
where $d_{\grk,I}$ is given by \eqref{eq:definition_de_dkI}. For each $I\subset [d]$ and each
 $\grk$ such that $
 \gr{0}\imd\gr{k}\imd\gr{n}$, the following inequalities take place:
 \begin{multline}
\frac 1{\abs{\gr{2^n}}^{1/2}\prod_{q=1}^d \pr{L\pr{n_q}}^{1/2}} \max_{\gr{1}-\gr{1}_{I}\imd\gri\imd 
\gr{2^{n-k}}}\abs{S_{\gri}^I\pr{T^{\gr{2^k}   },d_{\grk,I}   }  }\\ 
\leq\frac 1{\abs{\gr{2^k}}^{1/2}} \frac 1{\abs{\gr{2^{n-k}}}^{1/2}\prod_{q=1}^d \pr{L\pr{n_q-k_q}}^{1/2}} \max_{\gr{1}-\gr{1}_{I}\imd\gri\imd 
\gr{2^{n-k}}}\abs{S_{\gri}^I\pr{T^{\gr{2^k}   },d_{\grk,I}   }  }\\
\leq \frac 1{\abs{\gr{2^k}}^{1/2}}
\sup_{\grm \smd\gr{0}}\frac 1{\abs{\gr{2^{m}}}^{1/2}\prod_{q=1}^d \pr{L\pr{m_q}}^{1/2}} \max_{\gr{1}-\gr{1}_{I}\imd\gri\imd 
\gr{2^{m}}}\abs{S_{\gri}^I\pr{T^{\gr{2^k}   },d_{\grk,I}   }  },
 \end{multline}
 and combining with \eqref{eq:demo_MW_controle_de_Mf}, we derive that 
 \begin{equation}
 M\pr{f}\leq c_d \sum_{I\subset [d]}\sum_{\grk \smd\gr{0}}
 \frac 1{\abs{\gr{2^k}}^{1/2}}
\sup_{\grm \smd\gr{0}}\frac 1{\abs{\gr{2^{m}}}^{1/2}\prod_{q=1}^d \pr{L\pr{m_q}}^{1/2}} \max_{\gr{1}-\gr{1}_{I}\imd\gri\imd 
\gr{2^{m}}}\abs{S_{\gri}^I\pr{T^{\gr{2^k}   },d_{\grk,I}   }  }.
 \end{equation}
 Consequently, 
\begin{equation}
\norm{  M\pr{f}}_{p,w}  \leq c_d \sum_{I\subset [d]}\sum_{\grk \smd\gr{0}}
 \frac 1{\abs{\gr{2^k}}^{1/2}}c_{k,I}, 
\end{equation} 
 where 
\begin{equation}
c_{k,I}:=\norm{\sup_{\grm \smd\gr{0}}\frac 1{\abs{\gr{2^{m}}}^{1/2}\prod_{q=1}^d \pr{L\pr{m_q}}^{1/2}} \max_{\gr{1}-\gr{1}_{I}\imd\gri\imd 
\gr{2^{m}}}\abs{S_{\gri}^I\pr{T^{\gr{2^k}   },d_{\grk,I}   }  }}_{p,w}.
\end{equation} 
 We cannot directly apply Theorem~\ref{thm:LIL_orthomartingales} because of the 
 particular partial sums $S_{\gri}^I$ defined in \eqref{eq:sommes_partielles_dans_une_direction}. 
 First, it suffices to control $c_{k,I}$ when $I=[d]\setminus [i]$, $0\leq i\leq d$. The general case
 can be deduced form this one by permuting the roles of the operators $T^{\gr{e_q}}$. 
 
 We first consider the case where $i=d$. Then $I$ is the empty set and due to the definition 
 given by \eqref{eq:sommes_partielles_dans_une_direction}, there is no summation. Therefore, 
 \begin{equation}
c_{k,\emptyset}=\norm{\sup_{\grm \smd\gr{0}}\frac 1{\abs{\gr{2^{m}}}^{1/2}\
\prod_{q=1}^d \pr{L\pr{m_q}}^{1/2}} \max_{\gr{1}-\gr{1}_{I}\imd\gri\imd 
\gr{2^{m}}}\abs{ d_{\grk,I} \circ T^{\gr{2^k\cdot \gri }  }  }}_{p,w},
\end{equation} 
 and for a fixed $x$, 
 \begin{equation}
 \PP\ens{\sup_{\grm \smd\gr{0}}\frac 1{\abs{\gr{2^{m}}}^{1/2}}
\prod_{q=1}^d \pr{L\pr{m_q}}^{1/2} \max_{\gr{1}-\gr{1}_{I}\imd\gri\imd 
\gr{2^{m}}}\abs{ d_{\grk,\emptyset} \circ T^{\gr{2^k\cdot \gri }  }  }>x}\leq 
\sum_{\grm \smd\gr{0}}\abs{\gr{2^{m}}}\PP\ens{ 
\abs{ d_{\grk,\emptyset}}>x\abs{\gr{2^{m}}}^{1/2}
}.
 \end{equation}
Now, taking into account the fact for a fixed $k$, the number of 
elements of $\N^d$ such that whose sum is $N$ is bounded by $C_dN^{d-1}$, 
an application of \eqref{eq:somme_2puisk_kpuisq} in Lemma~\ref{lem:control_somme_probas_par_Lpq} 
gives that 
\begin{equation}
c_{k,\emptyset}\leq \norm{d_{\grk,\emptyset}}_{2,d-1}
\end{equation}

 Assume now that  $1\leq i\leq d-1$.    Let 
 \begin{equation}
 Y:= \sup_{m_{i+1},\dots,m_d\geq 0}2^{-\frac 12\sum_{q=i+1}^dm_q}\prod_{q=i+1}^d 
 \pr{L\pr{m_q}}^{-1/2}
 \max_{j_{i+1},\dots,j_d}\abs{\sum_{\ell_{i+1}=0}^{j_{i+1}-1}   \dots 
 \sum_{\ell_{d}=0}^{j_{d}-1}  U^{\sum_{u=i+1}^d 
 2^{k_{u}}\ell_{u}\gr{e_{u}}      }       d_{\grk,I}   } .
\end{equation}  
 With this notation, the inequality 
\begin{equation}
c_{k,I}\leq 
\norm{\sup_{m_1,\dots,m_i\geq 0} 2^{-\frac{m_1+\dots+m_i}2}
\prod_{q=1}^iL\pr{m_i}^{-1/2} \max_{ 0\leq  i_q\leq 2^{m_q}, 1\leq q\leq i
 } Y \circ T^{\sum_{q=1}^i i_q\gr{e_q}  }    }_{p,w}
\end{equation} 
holds 
and with the same arguments as before, it follows that 
\begin{equation}
c_{k,I}\leq  C_{p,d}\norm{Y}_{2,i-1}. 
\end{equation}
Now, we can apply Theorem~\ref{thm:LIL_orthomartingales} to $\widetilde{d}:=d- i$, 
$m=   d_{\grk,I} $, $\widetilde{T}^{\gri}=T^{\gr{2^k}\gr{i}}$ and $r=i-1$ to get that 
\begin{equation}
c_{k,I}\leq  C_{p,d}\norm{ d_{\grk,I}  }_{2,i-1+2\pr{d-i}}
=  C_{p,d}\norm{ d_{\grk,I}  }_{2,2d-i-1}\leq  C_{p,d}\norm{ d_{\grk,I}  }_{2,2\pr{d-1}}.
\end{equation}
When $I=[d]$ we can directly apply Theorem~\ref{thm:LIL_orthomartingales}. In total, we get that 
\begin{equation}
\norm{  M\pr{f}}_{p,w} \leq c_{p,d} \sum_{\grk \smd \gr{0}} \sum_{I\subset [d]}
\abs{\gr{2^k}  }^{-1/2}\norm{ d_{\grk,I}  }_{2,2\pr{d-1}}.
\end{equation}
Now, keeping in mind the definition of $d_{\grk,I}$ given by \eqref{eq:definition_de_dkI}, the 
following inequality takes place
\begin{equation}
\norm{ d_{\grk,I}  }_{2,2\pr{d-1}}\leq \abs{I}^2 
\norm{\proj{S_{2^{\grk+\gr{1_{Z\pr{k}}}-\gr{1_I}}}\pr{f}  }{0} }_{2,2\pr{d-1}}.
\end{equation}
Now, using the fact that $\norm{\proj{S_{2^{\grk+\gr{1_{Z\pr{k}}}-\gr{1_I}}}\pr{f}  }{0} }_{2,2\pr{d-1}}
\leq c_d\norm{\proj{S_{2^{\grk  }}\pr{f}  }{0} }_{2,2\pr{d-1}}+c_d \norm{f}_{2,2\pr{d-1}}$, 
we derive that 
\begin{equation}
\norm{  M\pr{f}}_{p,w}\leq c_{p,d}   \sum_{\grk \smd \gr{0}}
\abs{\gr{2^k}  }^{-1/2}
\norm{\proj{S_{2^{\grk  }}\pr{f}  }{0} }_{2,2\pr{d-1}}.
\end{equation}
Now, we have to bound the series in the right hand side of the previous equation in terms 
of right hand side of \eqref{eq:inequalite_LLI_MW}. To this aim, we define for fixed 
$k_1,\dots,k_{d-1}\geq 1$ the quantity 
\begin{equation}
V_n^{\pr{d}}:=\norm{ \proj{S_{2^{k_1},\dots,2^{k_{d-1}},n}\pr{f}}{0}}_{2,2\pr{d-1}}.
\end{equation}
Then the sequence $\pr{V_n^{\pr{d}}}_{n\geq 1}$ is subadditive. Therefore, 
by Lemma~2.7 in \cite{MR2123210}, 
\begin{multline}
 \sum_{\grk \smd \gr{0}}
\abs{\gr{2^k}  }^{-1/2}
\norm{\proj{S_{2^{\grk  }}\pr{f}  }{0} }_{2,2\pr{d-1}}\\ \leq 
C_d \sum_{k_1,\dots,k_{d-1}}\sum_{n\geq 1}2^{-\frac 12\pr{k_1+\dots+k_{d-1}}}
\frac 1{n_d^{3/2}}\norm{ \proj{S_{2^{k_1},\dots,2^{k_{d-1}},n_d}\pr{f}}{0}}_{2,2\pr{d-1}}.
\end{multline}
Then defining for a fixed $n_d$ and fixed $k_1,\dots,k_{d-2}$ the sequence 
\begin{equation}
V_n^{\pr{d-1}}:=\norm{ \proj{S_{2^{k_1},\dots,2^{k_{d-2}},n_{d-1},n_d}\pr{f}}{0}}_{2,2\pr{d-1}},
\end{equation}
we get an other subadditive sequence. By repeating this argument, we 
end the proof of Theorem~\ref{thm:MW_condition}.
\end{proof}
\begin{proof}[Proof of Corollary~\ref{cor:application_champs_lineaires_cond_proj}]
Observe that $\pi_{\grj}\pr{f}=a_{\grj}m$, hence \eqref{eq:Hannan_processus_lineaires} 
follows.

In order to prove \eqref{eq:MW_processus_lineaires}, we first have to simplify 
$\proj{S_n\pr{f}}{\gr{0}}$. First, 
\begin{equation}
S_{\grn}\pr{f}=\sum_{\gr{0}\imd \gri\imd \grn-\gr{1}}\sum_{\grj\smd \gr{0}}
a_{\grj} m\circ T^{\gri-\grj}
= \sum_{\gr{0}\imd \gri\imd \grn-\gr{1}}\sum_{\gr{\ell}\imd 
\gr{i}}
a_{ \gri - \gr{\ell}  } m\circ T^{\gr{\ell}}.
\end{equation}
 
By conditioning with respect to $\Fca_{\gr{0}}$, only the terms with index $\gr{\ell}\imd\gr{0}$ 
remain. Hence 
\begin{equation}
\proj{S_n\pr{f}}{\gr{0}}=\sum_{\gr{0}\imd \gri\imd \grn-\gr{1}}\sum_{\gr{\ell}\imd 
\gr{0}}
a_{ \gri - \gr{\ell}  } m\circ T^{\gr{\ell}}.
\end{equation}
By using a combination of Lemmas~3.1 and 6.1 in \cite{MR0365692}, we derive that for all martingale 
difference sequence $\pr{d_j}_{j\geq 1}$, 
\begin{equation}
 \norm{ \sum_{j=1}^nd_j  }^2_{2,2\pr{d-1}}\leq 
 C_d \sum_{j=1}^n\norm{d_j  }^2_{2,2\pr{d-1}}.
\end{equation}
By induction on the dimension, this can be extended to sum 
of orthomartingales differences on a rectangle, then 
by the use of Fatou's lemma, we can apply this to summation 
on $\Z$. This gives 

\begin{equation}
\norm{\proj{S_n\pr{f}}{\gr{0}}}_{2,2\pr{d-1}}\leq  C_d 
\pr{
\sum_{\gr{\ell}\imd 
\gr{0}}\norm{   \sum_{\gr{0}\imd \gri\imd \grn-\gr{1}}
a_{ \gri - \gr{\ell}  } m\circ T^{\gr{\ell}} }_{2,2\pr{d-1}}
^2}^{1/2}.
\end{equation}
Since $T^{\gr{\ell}}$ is measure preserving, the following 
equality holds 
\begin{equation}
 \norm{   \sum_{\gr{0}\imd \gri\imd \grn-\gr{1}}
a_{ \gri - \gr{\ell}  } m\circ T^{\gr{\ell}} }_{2,2\pr{d-1}}^2
=\pr{\sum_{\gr{0}\imd \gri\imd \grn-\gr{1}}
a_{ \gri - \gr{\ell}  }}^2\norm{m}_{2,2\pr{d-1}}^2.
\end{equation}
We can conclude from Theorem~\ref{thm:MW_condition}.
This ends the proof of 
Corollary~\ref{cor:application_champs_lineaires_cond_proj}.
\end{proof}

\begin{proof}[Proof of Corollary~\ref{cor:application_champs_multi_lineaires_cond_proj}]
It suffices to prove the representation of $f$ given by \eqref{eq:representation_de_f}; then 
inequalities \eqref{eq:Hannan_processus_multi_lineaires} and \eqref{eq:MW_processus_multi_lineaires} follow 
from an application of Corollary~\ref{cor:application_champs_lineaires_cond_proj} to each linear process involved in 
\eqref{eq:representation_de_f}. The assumptions imply that 
\begin{equation}\label{eq:dec_of_f_as_proj}
f=\sum_{\grj\smd \gr{0}} \pi_{-\grj}\pr{f},
\end{equation}
where $\pi_{-\grj}$ is defined by \eqref{eq:definition_projectors}. Observe that 
$U^{\grj}\pi_{-\grj}\pr{f}$ belongs to the space $W_d$ defined by \eqref{eq:dfn_espace_W_d}. Therefore, this function admits 
the representation 
\begin{equation}
U^{\grj}\pi_{-\grj}\pr{f}=\sum_{k=1}^{+\infty}a_{k,\grj}\pr{f} e_k.
\end{equation}
Plugging this equality in \eqref{eq:dec_of_f_as_proj} gives \eqref{eq:representation_de_f}.
\end{proof}
\textbf{Acknowledgement} Research supported by the DFG Collaborative Research Center SFB 823 `Statistical modelling of nonlinear dynamic processes'.

The author would like to thank the referee for many comments that improved 
the readability of the paper.


\def\polhk\#1{\setbox0=\hbox{\#1}{{\o}oalign{\hidewidth
  \lower1.5ex\hbox{`}\hidewidth\crcr\unhbox0}}}\def\cprime{$'$}
  \def\polhk#1{\setbox0=\hbox{#1}{\ooalign{\hidewidth
  \lower1.5ex\hbox{`}\hidewidth\crcr\unhbox0}}} \def\cprime{$'$}
\providecommand{\bysame}{\leavevmode\hbox to3em{\hrulefill}\thinspace}
\providecommand{\MR}{\relax\ifhmode\unskip\space\fi MR }
\providecommand{\MRhref}[2]{%
  \href{http://www.ams.org/mathscinet-getitem?mr=#1}{#2}
}
\providecommand{\href}[2]{#2}


\begin{thebibliography}{CDV15}

\bibitem[BT08]{MR2462551}
B. Bercu and A. Touati, \emph{Exponential inequalities for
  self-normalized martingales with applications}, Ann. Appl. Probab.
  \textbf{18} (2008), no.~5, 1848--1869. \MR{2462551}

\bibitem[Bur73]{MR0365692}
D.~L. Burkholder, \emph{Distribution function inequalities for martingales},
  Ann. Probab. \textbf{1} (1973), 19--42. \MR{0365692 (51 \#1944)}

\bibitem[CDV15]{MR3504508}
Ch. Cuny, J.~Dedecker, and D.~Voln\'{y}, \emph{A functional {CLT} for fields of
  commuting transformations via martingale approximation}, Zap. Nauchn. Sem.
  S.-Peterburg. Otdel. Mat. Inst. Steklov. (POMI) \textbf{441} (2015),
  no.~Veroyatnost\cprime i Statistika. 22, 239--262. \MR{3504508}

\bibitem[Cun15]{MR3322323}
Ch. Cuny, \emph{A compact {LIL} for martingales in 2-smooth {B}anach
  spaces with applications}, Bernoulli \textbf{21} (2015), no.~1, 374--400.
  \MR{3322323}

\bibitem[Cun17]{MR3650410}
\bysame, \emph{Invariance principles under the {M}axwell-{W}oodroofe condition
  in {B}anach spaces}, Ann. Probab. \textbf{45} (2017), no.~3, 1578--1611.
  \MR{3650410}

\bibitem[FGL15]{MR3311214}
X. Fan, I. Grama, and Q. Liu, \emph{Exponential inequalities for
  martingales with applications}, Electron. J. Probab. \textbf{20} (2015), no.
  1, 22. \MR{3311214}

\bibitem[FGL17]{MR3734020}
\bysame, \emph{Martingale inequalities of type {D}zhaparidze and van {Z}anten},
  Statistics \textbf{51} (2017), no.~6, 1200--1213. \MR{3734020}

\bibitem[Gir18]{MR3869881}
D. Giraudo, \emph{Invariance principle via orthomartingale approximation},
  Stoch. Dyn. \textbf{18} (2018), no.~6, 1850043, 29. \MR{3869881}

\bibitem[Gir19a]{giraudo2019bounded}
D. Giraudo, \emph{Bound on the maximal function associated to the law of
  the iterated logarithms for {B}ernoulli random fields}, 2019.

\bibitem[Gir19b]{giraudo2019_deviation} 
D. Giraudo, \emph{Deviation inequalities for {B}anach space valued martingales
              differences sequences and random fields}, ESAIM Probab. Stat. 
   \textbf{23} (2019), 922--946. \MR{4046858},
 
\bibitem[Jia99]{MR1674964}
J. Jiang, \emph{Some laws of the iterated logarithm for two parameter
  martingales}, J. Theoret. Probab. \textbf{12} (1999), no.~1, 49--74.
  \MR{1674964 (2000e:60046)}

\bibitem[Kre85]{MR797411}
U. Krengel, \emph{Ergodic theorems}, De Gruyter Studies in Mathematics,
  vol.~6, Walter de Gruyter \& Co., Berlin, 1985, With a supplement by Antoine
  Brunel. \MR{797411}

\bibitem[Pis76]{MR0501237}
G. Pisier, \emph{Sur la loi du logarithme it\'{e}r\'{e} dans les espaces de
  {B}anach}, Probability in {B}anach spaces ({P}roc. {F}irst {I}nternat.
  {C}onf., {O}berwolfach, 1975), 1976, pp.~203--210. Lecture Notes in Math.,
  Vol. 526. \MR{0501237}

\bibitem[PU05]{MR2123210}
M. Peligrad and Sergey Utev, \emph{A new maximal inequality and invariance
  principle for stationary sequences}, Ann. Probab. \textbf{33} (2005), no.~2,
  798--815. \MR{2123210 (2005m:60047)}

\bibitem[PZ18a]{MR3798239}
M. Peligrad and Na~Zhang, \emph{Martingale approximations for random
  fields}, Electron. Commun. Probab. \textbf{23} (2018), Paper No. 28, 9.
  \MR{3798239}

\bibitem[PZ18b]{MR3769664}
\bysame, \emph{On the normal approximation for random fields via martingale
  methods}, Stochastic Process. Appl. \textbf{128} (2018), no.~4, 1333--1346.
  \MR{3769664}

\bibitem[Vol15]{MR3427925}
D. Voln{\'y}, \emph{A central limit theorem for fields of martingale
  differences}, C. R. Math. Acad. Sci. Paris \textbf{353} (2015), no.~12,
  1159--1163. \MR{3427925}

\bibitem[Vol18]{VOLNY2018}
D. Volný, \emph{On limit theorems for fields of martingale differences},
  Stochastic Process. Appl. \textbf{129} (2019), no.~3, 841--859. \MR{3913270}

\bibitem[VW14]{MR3264437}
D. Voln\'{y} and Yizao Wang, \emph{An invariance principle for stationary
  random fields under {H}annan's condition}, Stochastic Process. Appl.
  \textbf{124} (2014), no.~12, 4012--4029. \MR{3264437}

\bibitem[Wic73]{MR0394894}
M.~J. Wichura, \emph{Some {S}trassen-type laws of the iterated logarithm
  for multiparameter stochastic processes with independent increments}, Ann.
  Probab. \textbf{1} (1973), 272--296. \MR{0394894 (52 \#15693)}

\bibitem[WW13]{MR3222815}
Y. Wang and M. Woodroofe, \emph{A new condition for the invariance
  principle for stationary random fields}, Statist. Sinica \textbf{23} (2013),
  no.~4, 1673--1696. \MR{3222815}

\bibitem[ZRP18]{2018arXiv180908686Z}
N.~{Zhang}, L.~{Reding}, and M.~{Peligrad}, \emph{{On the quenched CLT for
  stationary random fields under projective criteria}}, ArXiv e-prints (2018).

\end{thebibliography}
\end{document}